[1994/12/01]% LaTeX date must December 1994 or later
\documentclass{ijmart}
\chardef\bslash=`\\ % p. 424, TeXbook
\theoremstyle{definition}

\theoremstyle{remark}

\def\N{\mathbb N}
\def\Z{\mathbb Z}	
\def\Q{\mathbb Q}
\def\R{\mathbb R}
\def\C{\mathbb C}

\theoremstyle{definition}
	\newtheorem{definition}{Definition}[section]

    \newtheorem{remark}[definition]{Remark}

\theoremstyle{remark}
	
    \newtheorem{lemma}[definition]{Lemma}
	\newtheorem{proposition}[definition]{Proposition}
	\newtheorem{theorem}[definition]{Theorem}
	\newtheorem{corollary}[definition]{Corollary}

\usepackage{amsmath}
\usepackage{amsthm}
\usepackage{amsfonts}
\usepackage{amssymb}
\usepackage{mathrsfs}
\usepackage{amstext}
\usepackage{amsopn}
\usepackage{graphicx}
\usepackage{subfig}
\usepackage{multirow}
\usepackage{color}
\usepackage{float}
\usepackage{pdflscape}

\usepackage[counterclockwise, figuresright]{rotating}
\newcommand{\eval}[2][\right]{\relax
  \ifx#1\right\relax \left.\fi#2#1\rvert}

\begin{document}
\title[\quad Factorizable  ODEs and Hayman's Conjecture]{Nonlinear  Loewy Factorizable Algebraic ODEs and Hayman's Conjecture}
\author[T.W. Ng]{Tuen-Wai Ng\thanks{Supported by PROCORE - France/Hong Kong joint research grant, F-HK39/11T
 and RGC grant HKU 704409P.}}
\address{Department of Mathematics,
The University of Hong Kong\\
Pokfulam Road,
Hong Kong}
\email{ntw@maths.hku.hk}
%\author[Magnes Press]{Cheng-Fa Wu\thanks{This
 %   is the copyright owner of the style}}
\author[C.F. Wu]{Cheng-Fa Wu\thanks{Supported by  NSFC grant (No. 11701382) and RGC grant  HKU 704611P.}}
\address{Institute for Advanced Study, Shenzhen University \\ Shenzhen, PR China}
\email{cfwu@szu.edu.cn}
%\date{September 23, 2017}

%\date{Received on March 20, 2017}
%\issueinfo{VOL}{NUM}{MONTH}{YEAR}
%\issueinfo{000}{NUM}{MONTH}{0000}
%\doiinfo{10.1007/DOI-NUMBER}
\begin{abstract}
  In this paper, we introduce certain $n$-th order nonlinear Loewy factorizable
algebraic ordinary differential equations for the first time and study the growth of their meromorphic solutions in terms of the Nevanlinna
characteristic function. It is shown that for
generic cases all their meromorphic solutions are elliptic functions or
their degenerations and hence their order of growth are at most two. Moreover, for the second order factorizable algebraic
ODEs, all the meromorphic solutions of them (except for one case) are
found explicitly. This allows us to show that a conjecture proposed by Hayman in 1996 holds for these second order ODEs.
\end{abstract}
\maketitle

{\it \noindent Dedicated to Professor Walter Hayman on the occasion of his 90th birthday}

\section{Introduction}

 One important aspect  of  the studies of complex differential equations is to investigate the growth of their  solutions which are meromorphic on the whole complex plane.
 %Throughout this paper, meromorphic function refers to a function meromorphic on the whole complex plane.
 %Nevertheless, even for some special types of differential equations, there are many problems in this direction which have not been solved   yet.
  A well known problem in this direction is the following conjecture proposed by Hayman in \cite{Hayman1996} (see also  \cite[p.~344]{Laine2008}). It is also referred as the {\it classical conjecture} in \cite{Chiang2003Halburd}.
%\begin{conjecture}[Hayman] \label{Hayman's Classical Conjecture}

%\eqref{algebraic ordinary differential equation-factoriazble ODEs}

\noindent {\bf Hayman's conjecture for algebraic ODEs :}
 If $f$ is a meromorphic solution of
\begin{equation*} \label{algebraic ordinary differential equation-factoriazble ODEs}
P(z,f,f',\cdots,f^{(n)})=0,
\end{equation*}
where $P$ is a polynomial in all its arguments, then there exist $a,b,c \in \R^{+}$ such that
\begin{equation} \label{classical conjecture-factoriazble ODEs}
T(r,f) < a \exp_{n-1}(br^c), 0 \leq r < \infty,
\end{equation}
where   $T(r, f)$ is the Nevanlinna characteristic of $f(z)$ and $\exp_l(x)$ is the $l$ times iterated exponential, i.e.,
\begin{equation*}
\exp_0(x) = x, \exp_1(x) = e^x, \exp_l(x) = \exp \{\exp_{l-1}(x)\}.
\end{equation*}

%{\color{red} It is due to Bank  for the case $n = 2$. Also, it has been listed as an open problem by Eremenko \cite[p.~491]{Barth1984BrannanHayman}
%   and Rubel  for entire solutions of $n$th order equations.}

%{\color{blue}

Note that the conjecture is due to Bank \cite{Bank1975} for the case $n = 2$. Also, it
has been listed as an open problem by Eremenko \cite[p.~491]{Barth1984BrannanHayman} and Rubel \cite[p.~662]{Rubel1992} for the case of entire solutions.

This conjecture is closely related to a false conjecture due to E. Borel on the growth of real-valued solutions. In \cite{Borel1899}, Borel proved that any real-valued solution defined on the interval $(x_0, \infty)$ of the first-order algebraic ODE is dominated by $\exp_2(x)$ for all sufficiently large $x$ (improvements of this result were later made by Lindel\"{o}f \cite{Lindelof1899} and Hardy \cite{Hardy1912}). In the same paper, Borel dealt with higher-order ODEs as well and showed that all such solutions of $n$-th order algebraic ODEs are eventually dominated by $\exp_{n+1}(x)$. However, it was later pointed out by Fowler \cite{Fowler1914}, Vijayaraghavan \cite{Vijayaraghavan1932} etc. that Borel's proof in the higher-order case was incorrect. The counter-examples constructed by Vijayaraghavan etc. \cite{Vijayaraghavan1932,Basu1937BoseVijayaraghavan} demonstrate that second-order algebraic ODEs may possess real-valued solutions dominating any given increasing function  for a sequence of points tending to $\infty$.

Hayman's conjecture is true when $n=1$ by a result of Gol'dberg \cite{Goldberg1956}
while it is still open for any $n \geq 2$. For general $n \in \N$, it was proved by Eremenko, Liao and Ng \cite{Eremenko2009LiaoNg} that the conjecture is true for the ODE $P(f^{(n)}, f) = 0$, where $P \in \C [x,y]$ is a non-constant polynomial. In fact, they proved that any meromorphic solution with at least one pole must be an elliptic function or its degenerations. For some other partial results of
this conjecture, we refer the readers to \cite{Bank1975,Conte2015NgWu,Gallagher2015,Halburd2015Wang,Hayman1996} and the references
therein.
%}

%The conjecture for the case $n = 2$ was proposed by Bank in \cite{Bank1975}. When $n=1$, \eqref{classical conjecture} reduces to
%\begin{equation} \label{finite order of growth}
%T(r,f) < a \cdot br^c, 0 \leq r < \infty,
%\end{equation}
%and we say that a meromorphic function $f$ has finite order if it satisfies
%\eqref{finite order of growth}. The infimum $\sigma$ of all possible number $c$ is called the order of $f$. For results on this conjecture,  one may refer
%to \cite{Goldberg1956,Steinmetz1980,Hayman1996,Chiang2003Halburd,Laine2008} and the references therein.
 Since Hayman's conjecture seems to be out of
reach currently, we introduce and study the following {\it factorizable} $n$-th order algebraic ODE
\begin{eqnarray} \label{general factorizable equation}
[D - f_n(u)] \cdots [D - f_2(u)] [D - f_1(u)] (u - \alpha) = 0,
\end{eqnarray}
where $u = u(z), D = \dfrac{d}{d z}, \alpha \in \C$ and $f_i \in \C[x]$ ($i = 1, 2, \dots, n, n\in \N$).

%The ODE \eqref{general factorizable equation} is connected to the Loewy decomposition (Theorem \ref{Loewy's theorem})
 As we will see later, the study of ODE \eqref{general factorizable equation} is motivated by the consideration of Lowey decomposition (Theorem \ref{Loewy's theorem})  to linear ODEs (see \cite{Schwarz2012,Schwarz2013} and the references therein)  and it also covers some interesting well-known ODEs.
%The ODE \eqref{general factorizable equation} looks artificial but as we will see below, it does cover some interesting well-known ODEs.
Another reason to consider it is that the particular meromorphic solutions of \eqref{general factorizable equation} can have a tower structure because a solution of $[D - f_k(u)] \cdots [D - f_1(u)] (u - \alpha) = 0$ will also be a solution of $[D - f_{k+1}(u)] [D - f_{k}(u)]\cdots [D - f_1(u)] (u - \alpha) = 0$ for $k=1,...,n-1$ and it seems that one can get meromorphic solutions which grow faster and faster and eventually produce solutions which show that the estimate in \eqref{classical conjecture-factoriazble ODEs} is sharp. This is at least the case when $n=2$ and all $f_i$ are linear polynomials (see Remark \ref{2factors-meormorphic solutions outside W}).

A special case for the ODE \eqref{general factorizable equation} is that all the $f_i$ are constants, for which the equation  \eqref{general factorizable equation} becomes linear and thus Hayman's conjecture holds.
On the other hand, according to the following proposition, any linear ODE with constant coefficients can be rewritten in the form \eqref{general factorizable equation}.
% {\color{red} As we will see later, the study of ODE \eqref{general factorizable equation} is motivated by the consideration of Lowey decomposition to linear ODEs (see \cite{Schwarz2012,Schwarz2013} and the references therein). }
%shows us more, that is each linear ODE with constant coefficients can be decomposed into the form of \eqref{general factorizable equation} with $f_i$ being constants.
\begin{proposition} \label{factorization of linear equation}
For any $n \in \N$, the linear ODE
\begin{eqnarray} \label{general linear equation}
u^{(n)}(z) + c_{n-1} u^{(n - 1)}(z) + \cdots  + c_0 = 0, \quad c_i \in \C, i = 1,2, \dots, n
\end{eqnarray}
can be decomposed into the form
\begin{eqnarray} \label{general factorizable linear equation}
[D - b_n] \cdots [D - b_2] [D - b_1] (u - \alpha) = 0, \; D = \dfrac{d}{d z},
\end{eqnarray}
for some $\alpha, b_k \in \C$.
\end{proposition}
\noindent Proposition \ref{factorization of linear equation} is connected to a special case of Loewy  decomposition (Theorem \ref{Loewy's theorem}) and Corollary \ref{Loewy's corollary}. To state them, we recall some terminologies. A differential operator $L$ of
order $n$ is defined by
\begin{eqnarray} \label{operator}
 L : = D^n + r_{n-1} D^{n-1} + \cdots + r_1 D + r_0
\end{eqnarray}
where the coefficients $r_i , i = 1, \dots , n$, are rational functions over $\Q$, i.e., $r_i \in \Q(z)$.   $L$ is called {\it reducible} if it can be represented as the product of
two operators $L_1$ and $L_2$, i.e., $L = L_1 L_2$, both of order lower than $n$. In this case,
% $L = L_1 L_2$,
$L_1$ is called the {\it exact quotient}
of $L$ by $L_2$, and $L_2$ is called the {\it right factor} of $L$. Otherwise, the operator $L$ is called {\it irreducible}. The number of irreducible factors of $L$ in any two decompositions into irreducible factors is the same and any two such decompositions are linked by a permutation of the irreducible factors (see Proposition 1.1 of \cite{Schwarz2012}). It follows that there are only finitely many irreducible right factors of $L$. For any two operators $\widetilde{L}_1$ and $\widetilde{L}_2$, the {\it least common left multiple} denoted by $Lclm(\widetilde{L}_1, \widetilde{L}_2)$
is the operator of lowest order such that both $\widetilde{L}_1$ and $\widetilde{L}_2$ divide it from the right. An operator
which can be represented as $Lclm$ of irreducible operators   is called {\it completely reducible}.

Given $L$, consider all its irreducible right factors. Let $L_1^{(d_1)}$ be the $Lclm$ of all these irreducible right  factors and by construction $L_1^{(d_1)}$ is the unique completely reducible right factors of maximal order $d_1$. Factoring $L_1^{(d_1)}$ out from $L$ and repeating the same procedure to the remaining left factor of $L$, we have the following.
 % with $d_i = 1, m = n$.
\begin{theorem}  \label{Loewy's theorem} (\cite[p.~4]{Loewy1906,Schwarz2012})
Let $L$  be an operator of order $n$ as defined in  \eqref{operator}, then it can be uniquely written as the product of completely reducible factors $L^{(d_k)}_k$
of maximal order $d_k$ over $\Q(z)$ of the form
\begin{eqnarray*}
L = L^{(d_m)}_m L^{(d_{m-1})}_{m-1} \cdots L^{(d_1)}_1,
\end{eqnarray*}
where $d_1 + \cdots + d_m = n$. %and the factors $L^{(d_k)}_k$ are unique.
\end{theorem}
%\begin{theorem}[\cite{Loewy1906,Schwarz2013}]  \label{Loewy's theorem}
%Let $L$  be an operator of order $n$ as defined in  \eqref{operator}, then it can be uniquely written as the product of reducible factors $L^{(d_k)}_k$
%of maximal order $d_k$ over $\Q(z)$ of the form
%\begin{eqnarray*}
%% \nonumber to remove numbering (before each equation)
%L = L^{(d_m)}_m L^{(d_{m-1})}_{m-1} \cdots L^{(d_1)}_1
%\end{eqnarray*}
%where $d_1 + \cdots + d_m = n$. %and the factors $L^{(d_k)}_k$ are unique.
%\end{theorem}

\begin{corollary}[\cite{Schwarz2013}]  \label{Loewy's corollary}
Each factor $L^{(d_k)}_k, k =1,2, \dots, m$, in Theorem \ref{Loewy's theorem} can be expressed as
\begin{eqnarray*}
L^{(d_k)}_k = Lclm( l^{(e_1)}_{j_1}, l^{(e_2)}_{j_2}, \dots l^{(e_k)}_{j_k}),
\end{eqnarray*}
where $e_1 + \cdots + e_k = d_k$ and each $l^{(e_i)}_{j_i}, i = 1, \dots, k$, is an irreducible operator
of order $e_i$  over $\Q(z)$. % over Q(z).. %and the factors $L^{(d_k)}_k$ are unique.
\end{corollary}

The decomposition obtained in Theorem \ref{Loewy's theorem} is called the {\it Loewy decomposition}
of $L$ and it has been generalized to linear partial differential operators \cite{Tsarev2000}.  However, as far as we know, there is no similar study on Loewy decomposition for nonlinear ODEs. Therefore, we try to study nonlinear ODE of type  \eqref{deg1-general factorizable equation} below and we call   ODE \eqref{general factorizable equation} {\it nonlinear Loewy factorizable algebraic ODE}.

Among the non-linear cases of the equation \eqref{general factorizable equation}, the simplest case is perhaps the one with $\deg (f_j) \leq 1, i = 1, 2, \dots, n$, which we will study in this paper. In this case, we may assume $f_i = a_i u + b_i$, where $ a_i, b_i \in \C,i = 1, 2, \dots, n$. Then the equation \eqref{general factorizable equation} reduces to
\begin{equation} \label{deg1-general factorizable equation}
 [D - (a_n u + b_n)] \cdots [D - (a_2 u + b_2)] [D - (a_1 u + b_1)] (u - \alpha) = 0,  \; D = \dfrac{d}{d z},
\end{equation}
and our main results are as follows.
\begin{theorem} \label{mero sol of general factorizable equation}
%Assume deg$f_i \leq 1, i = 1, 2, \dots, n$, then f
For all $n \in \N$ and    $\mathbf{a} = (a_1, a_2, \dots, a_n) \in \C^n \backslash S$, where $S$ is the union of at most countably many hypersurfaces in $\C^n$, all meromorphic solutions (if they exist) of the ODE \eqref{deg1-general factorizable equation} belong to the class $W$, which consists of elliptic functions and their degenerations. Hence, for any generic $\mathbf{a} \in \C^n$, Hayman's conjecture is   true for \eqref{deg1-general factorizable equation}.
\end{theorem}
The proof of Theorem \ref{mero sol of general factorizable equation} is based on a long and careful application of Painlev\'{e} analysis as well as a simple application of Wiman-Valiron theory \cite[p.~51]{Laine1993}. We expect that this general method can also be used to show that for other types of non-linear algebraic ODEs with constant coefficients, a generic choice of the coefficients will make the corresponding ODE has all meromorphic solutions (if exist) in the class $W$.

 If $n = 1$, then the equation \eqref{general factorizable equation} is a particular Riccati equation and its meromorphic solutions can be easily derived, which are given by
\begin{eqnarray} \label{decomposition-2factors-paricular solution}
u(z) =
\begin{cases}
-\dfrac{\alpha +b_1 c e^{ \left(\alpha  a_1+b_1\right)z}}{a_1 c e^{ \left(\alpha  a_1+b_1\right)z}-1}, \alpha  a_1+b_1 \not = 0,
\\
\alpha-\dfrac{1}{a_1 z - c}, \alpha  a_1+b_1 = 0,
% a_1 \not = 0,
%\\
%c, a_1 = \alpha  a_1+b_1 = 0,
\end{cases}
c \, \, \text{arbitrary}.
\end{eqnarray}
Meanwhile, we can see from above that the Hayman's conjecture is sharp for $n = 1$.

Then the first non-trivial case for \eqref{general factorizable equation} is $n  = 2$, which has been studied in \cite{Cornejo-Perez2005Rosu}, and we will show that Hayman's  conjecture is
true for $n = 2$ apart from an exceptional case.
% Now the main results of this chapter can be stated as follows.
%For example, $\sin z, \cos z, \tan z, e^z $ and the gamma function $\Gamma (z)$ have order 1. The Weierstrass ellptic function $\wp(z)$, which satisfies $(\wp'(z))^2 = 4(\wp-e_1)(\wp-e_2)(\wp-e_3)$, has order 2.
%The function $e^{e^z}$ has infinite order but satisfies \eqref{classical conjecture} with $n=2, c=1$.

% In this paper, we will show that the {\it classical conjecture}  is true for the equation and
 %some special types of second order algebraic ODEs   which are not covered by any of the above results.

%\end{conjecture}

\begin{theorem} \label{solutions to 2-factorizble equations}

Consider the ordinary differential equation
\begin{eqnarray} \label{decomposition-2factors-1}
[D - f_2(u)] [D - f_1(u)] (u - \alpha) = 0,
\end{eqnarray}
where $u = u(z), D = \dfrac{d}{d z}, \alpha \in \C$ and $f_i(u) = a_i u + b_i, a_i, b_i \in \C, i = 1, 2$. If either $a_1 a_2 = 0$ or $2 - \dfrac{4a_1}{a_2}
\not \in \N \backslash \{1, 2, 3, 4, 6\}$,
%$(a_1, a_2) \in \C^2 \backslash S $,
%where $S = \{ (z, w) \in \C^2\} |   2-\frac{4z }{w}   \not \in \N \backslash \{1, 3, 4, 6\}, z w  \not = 0 \}$,
%\, \text{and} \,
%\begin{eqnarray}
%\begin{cases}
%a_1 a_2 \not = 0,
%\\
% 2-\frac{4a_1}{a_2}   \not \in \N \cup \{1, 3, 4, 6\},
%\end{cases}
%\end{eqnarray}
then \eqref{decomposition-2factors-paricular solution} is a particular meromorphic solution of the equation \eqref{decomposition-2factors-1}
and all other meromorphic solutions of \eqref{decomposition-2factors-1} are given in Table \ref{table: 2nd order factorizable-1}, \ref{table: 2nd order factorizable-2} and \ref{table: 2nd order factorizable-3} in the Appendix.
\end{theorem}

%\begin{remark}
%The traveling-wave reduction of one type of the well known KPP equation \cite{KolmogorovPetrovskiiPiskunov1937} is a special case of \eqref{decomposition-2factors-1}.
%
%\end{remark}

%{\color{blue}

\begin{remark}
 The ODE \eqref{decomposition-2factors-1}  reduces to the traveling wave reduction of the KPP equation \cite{Kolmogorov1937PetrovskyPiskunov} under certain choice of parameters.

\end{remark}
%}

\begin{remark}
For $n \geq 2$, the meromorphic solutions of \eqref{decomposition-2factors-1} given in Table \ref{table: 2nd order factorizable-1}, \ref{table: 2nd order factorizable-2} and \ref{table: 2nd order factorizable-3} are particular solutions of the ODE \eqref{deg1-general factorizable equation} as well.

\end{remark}

\begin{remark}
After normalization and expansion, the following case of equation \eqref{decomposition-2factors-1} remains unsolved
\begin{equation*}
 u'' +  (j - 4) u u' - (b_1 + b_2) u' + u \left(\frac{2-j}{2} u+b_1\right) \left(2 u+b_2\right) = 0, \quad j \in \N \backslash \{1, 2, 3, 4, 6\},
\end{equation*}
for which only particular meromorphic solutions have been found but not all of them.
\end{remark}

\begin{remark}
We note that equation \eqref{decomposition-2factors-1} (which is equivalent to equation \eqref{decomposition-2factors-2}) is a special case of the equation (G) in \cite[p.~326]{Ince1956}. In Ince's book \cite{Ince1956},  a classification of all equations of the form (G) such that all their solutions have no movable critical points is given. There,  except for a few simple cases, no explicit solutions have been given while here we are interested in constructing {\it all} meromorphic solutions of \eqref{decomposition-2factors-1} or \eqref{decomposition-2factors-2}. We would also like to emphasize that in the proof of Theorem \ref{solutions to 2-factorizble equations}, {\it Subcase A1} and {\it Subcase A4} correspond to the canonical form VI in \cite[p.~334]{Ince1956} whereas {\it Subcase A2} and {\it Subcase A3} correspond to the canonical form X in \cite[p.~334]{Ince1956}. No explicit solutions have been given for either of these two forms in \cite{Ince1956}. The readers can find these explicit solutions in Table 2. Table 1-3 may look scary but we think that they are of sufficient interest to applied mathematicians, physicists and engineers who are interested in explicit solutions of non-linear ODEs.

\end{remark}

%\begin{remark}
%By the transformation $u = k v, k = 1/a_2$, the equation \eqref{decomposition-2factors-1} becomes
%\begin{eqnarray} \label{decomposition-2factors-2}
%&& v'' +  (j - 4) v v' - (b_1 + b_2) v' + v \left(\frac{2-j}{2} v+b_1\right) \left(2 v+b_2\right) = 0, \quad j = 2 - \frac{4a_1}{a_2}.
%\end{eqnarray}

%\end{remark}

\begin{theorem} \label{growth estimate-2-factorizble equations}

%The first type of second order algebraic ODEs we consider is the following
With the same assumption on $a_1, a_2$ as given in Theorem \ref{solutions to 2-factorizble equations}, Hayman's conjecture holds for the equation \eqref{decomposition-2factors-1} and it is sharp in certain cases.
\end{theorem}

\begin{remark} \label{2factors-meormorphic solutions outside W}
From the Appendix, we see that the equation \eqref{decomposition-2factors-1} may have meromorphic solutions of the form
\begin{eqnarray*}
u_1(z) &=& - \dfrac{  q_i - q_k  } {2} e^{- \frac{  q_i - q_k }{\lambda} z}  \dfrac{\wp'(e^{- \frac{ q_i - q_k }{\lambda} z}  - \zeta_0; g_2, 0)}{\wp(e^{ - \frac{ q_i - q_k }{\lambda} z}  - \zeta_0; g_2, 0)} + q_k, \, \,  g_2 \in \C,
\\
u_2(z) &=&   \frac{\alpha  a_1-b_1}{2 a_1}-\sqrt{\dfrac{\beta}{a_1}} \dfrac{e^{\frac{b_2 z}{2}} \left(c_1 J_{\nu}'\left(\zeta \right)+c_2 Y_{\nu}'\left(\zeta \right)\right)}{  \left(c_1 J_{\nu}\left(\zeta \right)+c_2 Y_{\nu}\left(\zeta \right)\right)}, \, \text{or}
\\
u_3(z) &=& \alpha -\frac{\sqrt{2} b_1 c_0 e^{b_1 z} \tanh \left(\frac{1}{2} \left(\sqrt{2} c_0 e^{b_1 z}+c_1\right)\right)}{a_2}
\end{eqnarray*}
under some constraints on the parameters, and it will be shown in the proof of Theorem \ref{growth estimate-2-factorizble equations} that for $u_i, i = 1, 2, 3$, Hayman's conjecture is sharp for $n = 2$. Here,  $J_{\nu}\left(\zeta \right)$ and $Y_{\nu}\left(\zeta \right)$ are Bessel functions of the first and second kinds respectively.

\end{remark}
Finally, Remark \ref{2factors-meormorphic solutions outside W} shows that the ODE \eqref{decomposition-2factors-1} may have meromorphic  solutions outside the class $W$.
%Throughout the paper, we use the standard notations and results of Nevanlinna theory \cite{Hayman1964,Laine1993} and we shall consider solutions meromorphic on the whole complex plane.
%We shall also apply the Painlev\'{e} test   and we refer the readers to \cite{Conte1999} for details.
% Finally, class $W$ is defined to be a set which consists of elliptic functions
%and their successive degeneracies, i.e., elliptic functions, rational functions of one
%exponential $\exp(kz), k \in \C$ and rational functions of $z$.

 %Remark \ref{2factors-meormorphic solutions outside W} shows that the ODE \eqref{decomposition-2factors-1} may have meromorphic  solutions outside the class $W$.

%\clearpage

\section{Proof of Proposition \ref{factorization of linear equation}}

\begin{proof} We claim that for a fixed $n \in \N$, the characteristic equation of \eqref{general factorizable linear equation} is given by
%It is easy to see that the equation \eqref{general factorizable linear equation} is linear. Then w
\begin{equation*}
\Pi_{m = 1}^n (z - b_m)  = 0.
\end{equation*}
We prove this by induction.
Let $     \mathcal{A}_n = [D - b_n] \cdots [D - b_2] [D - b_1] (u - \alpha)$, then the claim holds obviously for $n = 1$. Assume it is true for $n = k$, then if $n = k+1$, as
\begin{eqnarray*}
\mathcal{A}_{k + 1} = \dfrac{d\mathcal{A}_k}{ d z} - b_{k+1} \mathcal{A}_k,
\end{eqnarray*}
the characteristic equation of the linear ODE $\mathcal{A}_{k + 1} = 0$ is
\begin{eqnarray*}
0 =  z \Pi_{m = 1}^k (z - b_m) - b_{k+1} \Pi_{m = 1}^k (z - b_m)
= \Pi_{m = 1}^{k + 1} (z - b_m).
\end{eqnarray*}
Next, for the equation  \eqref{general linear equation}, one may express its characteristic equation as
\begin{eqnarray*}
0 = z^n + c_{n -1} z^{n-1} + \cdots + c_0   = \Pi_{m = 1}^n (z - d_m), d_m \in \C.
\end{eqnarray*}
Consequently, one can rewrite the equation \eqref{general linear equation} as the form of \eqref{general factorizable linear equation} by choosing $\alpha, b_m \in \C,  1 \leq m \leq n$ such that $b_m = d_m$ and $(-1)^{n + 1} \alpha \Pi_{m = 1}^n   b_m =c_0  $ for $c_1 =(-1)^{n } \Pi_{m = 1}^n   b_m \not = 0$, otherwise $\alpha$ will be taken as an arbitrary constant.
\end{proof}

\begin{remark}
It is clear from the proof that the decomposition  \eqref{general factorizable linear equation} for the equation \eqref{general linear equation} is unique up to a permutation of the $b_m$'s.
\end{remark}

\section{Proof of Theorem \ref{mero sol of general factorizable equation}}

We first introduce some terminologies and notations.
 %and $q$ is the smallest integer among the list of leading powers.
Let  $I = (i_0, i_1, \dots, i_n), i_k \in \N \cup \{ 0\}, 0 \leq k \leq n$ and
% let $f = \chi^p$, then we have
 \begin{equation*}
H(y,y',\cdots,y^{(n)}) = \sum_{I \in \Lambda  } c_I y^{i_0} (y')^{i_1} \cdots (y^{(n)})^{i_n}, y = y(z), c_I  \in \C \backslash \{0\}.
\end{equation*}
If $y = z^p , -p \in \N$, then
 \begin{equation*}
H(y,y',\cdots,y^{(n)}) =  \sum_{I \in \Lambda } C_I z^{\alpha_I},
\end{equation*}
where $C_I \in \C, \alpha_I =  i_0 p + i_1 (p - 1) + \cdots + i_n (p - n)$.

Next, let $A$ be the set of those negative integers $p$
such that $\min_{I \in \Lambda} \alpha_I$ is attained by at least two $I$'s. Note that if $A = \emptyset$, then $H(y,y',\cdots,y^{(n)}) = 0$ has no meromorphic solutions with at least one pole. Suppose $A \not = \emptyset$, then for each $p \in A$, denote by $\Lambda' = \{I' \in \Lambda | \alpha_{I'} = \min_{I \in \Lambda} \alpha_I
\}$ and  we define the {\bf dominant terms} for each $p \in A$ to be
% possible pair $(p, q)$
 \begin{equation*}
\hat{E} = \sum_{I \in \Lambda'  } c_I y^{i_0} (y')^{i_1} \cdots (y^{(n)})^{i_n} .
\end{equation*}

Suppose $u(z) = \sum_{n=0}^{+\infty} u_nz^{n+p}(u_0 \not = 0, - p\in \N)$ with a pole at $z= 0$ is a meromorphic solution of $H(y,y',\cdots,y^{(n)})=0$.
%where $ H $ is a polynomial in $y$ and its derivatives with constant coefficients.
Then if we  plug $y = u(z)$ into $H$, we will get an expression of the form $E = \sum_{j = 0}^{+\infty}E_j z^{j+q} = 0, E_j \in \C$.
 % where $\chi = z-z_0, E_j \in \C$.
Since $y = u(z)$ is a solution of $H = 0$, we must have $E_j = 0$, for all $j \in \N \cup \{0\}$. Note that $E_0 = E_0(u_0;p)$ is a polynomial in $u_0$ with coefficients depending on $p$.

On the other hand, for $j = 1, 2, \dots$,
we can express $E_
j  $ as:
\begin{equation} \label{recursion-painleve test}
E_j \equiv P(u_0; j)u_j + Q_j(\{u_l | l < j\})  ,
\end{equation}
where $P(u_0; j)$ is a polynomial in $j$ determined by $u_0$ and $Q_j$ is a polynomial in $j$ with coefficients in $u_l ( l < j)$.
In fact, it is known that \cite{Darboux1883} (see also  \cite[p.~15]{Conte1999})
\begin{equation} \label{indicial equation}
P(u_0; j) = \lim_{z \rightarrow 0} z^{-j-q} \hat{E}'( u_0z^{p})z^{j+p},
\end{equation}
where $\hat{E}'(u)$ is defined by
\begin{equation*}
\hat{E}'(u) v :=  \lim_{\lambda \rightarrow 0} \dfrac{ \hat{E}(u + \lambda v) - \hat{E}(u) }{\lambda}.
\end{equation*}
%For each $j$, the above equation is linear in $u_j$.
%For the equation \eqref{recursion-painleve test} to vanish identically,
In order to have $E_j = 0$ for all $j \in \N$,
we must have for each $j$, either
\begin{itemize}
\item [1)] $u_j$ is uniquely determined by $P(u_0; j)$ and $Q_j$, {\it or}

\item [2)] both  $P(u_0; j)$ and $Q_j$  vanish,

\end{itemize}
otherwise there is no meromorphic function satisfying $H(y,y',\cdots,y^{(n)}) = 0$.

Therefore if the polynomial $P(u_0; j)$ in $j$ does not have any nonnegative integer root, then each $u_j$ is uniquely determined by $P(u_0; j)$ and $Q_j$.

 \begin{definition}
 The zeros of $P(u_0; j)$ are defined to be the {\bf Fuchs indices} of the equation $H(y,y',\cdots,y^{(n)}) = 0$ and the {\bf indicial equation} is defined as  $P(u_0; j) = 0$.
\end{definition}

From the above definition and \eqref{indicial equation}, one sees that the Fuchs indices of an ODE are determined by its dominant terms and the values of $u_0$. Therefore, to compute the Fuchs indices of the ODE \eqref{deg1-general factorizable equation}, we have to find its dominant terms, denoted by $\hat{E}_n$, and $u_0$. We will see that any terms involving $b_i$'s will not be included  in the dominant terms when all the $a_i$'s are non-zero. Therefore, it would be useful to first look at
\begin{eqnarray}
\mathfrak{D}_n = \mathfrak{D}_n (u(z)) := [D - a_n u] \cdots [D - a_2 u] [D - a_1 u] u.
\end{eqnarray}
Then we may express $\mathfrak{D}_n$ as
\begin{eqnarray}
\mathfrak{D}_n
&=& \sum_{I \in \Omega} c_{I} u^{i_0} (u')^{i_1} \cdots (u^{(n)})^{i_n} \label{degree-general factorizable equation}
\\
&=&  u^{(n)} + (-1)^n \Pi_{k = 1}^n a_k u^{n+1} + \cdots, \nonumber
\end{eqnarray}
where $c_I \in \C, i_\kappa \in \N \cup \{0\}, \kappa = 0, 1, \dots, n$,
and we have
\begin{lemma} \label{weight-general factorizable equation}
For any $(i_0, i_1, \dots, i_n) \in \Omega$ in \eqref{degree-general factorizable equation}, we have
\begin{eqnarray*}
i_0 + 2i_1 + \cdots + (n + 1)i_n = n+1.
\end{eqnarray*}
\end{lemma}

\begin{proof}
  We prove by induction.
 It is obvious for $n = 1$.
Now suppose $i_0 + 2i_1 + \cdots + (n + 1)i_n = n+1$, then% for $n = K$, then for $n = K+1$,
\begin{eqnarray} \label{recursion1-general factorizable equation}
{\color{white} aaa}\mathfrak{D}_{n+1} &=& [D - a_{n+1} u] \mathfrak{D}_n
\\
&=& \dfrac{d \mathfrak{D}_n}{d z} - a_{n+1} u  \mathfrak{D}_n \nonumber
\\
&=& \sum_{(i_0, i_1, \dots, i_n )} c_{i_0, i_1, \dots, i_n} \left[ \sum_{k=0}^n i_k u^{i_0} (u')^{i_1} \cdots (u^{(n)})^{i_n}  \dfrac{u^{(k+1)}}{u^{(k)}} \right] \nonumber
\\
&& - \sum_{(i_0, i_1, \dots, i_n )} a_{n+1} c_{i_0, i_1, \dots, i_n}    u^{i_0 + 1} (u')^{i_1} \cdots (u^{(n)})^{i_n} \nonumber
\\
&=&  \sum_{(j_0, j_1, \dots, j_n, j_{n+1} )}   d_{j_0, j_1, \dots, j_n,  j_{n+1}  }    u^{j_0 } (u')^{j_1} \cdots (u^{(n)})^{j_n} (u^{(n + 1)})^{j_{n + 1 }} . \nonumber
\end{eqnarray}
Here $(j_0, j_1, \dots, j_n, j_{n+1} ) = (i_0 + 1, i_1, \dots, i_n, 0 )$ or $(i_0, i_1, \dots, i_k - 1, i_{k+1} + 1, \dots )$, and in both cases we have $j_0 + 2 j_1 + \cdots +(n+1) j_n + (n+2) j_{n+1} = n+2$.

\end{proof}

\begin{lemma} \label{order+lead coe-general factorizable equation}
Suppose $u(z) =  \sum_{r = 0}^{\infty} u_r z^{r+ p} \, ( u_0 \not = 0, -p \in \N)$ is a meromorphic solution of the ODE \eqref{deg1-general factorizable equation}
with all the $a_i \not = 0$, then for any $n \in \N$
%has the following properties
\begin{itemize}
%\item [(i)]

\item [(i)] $p = -1$.

\item [(ii)] The dominant terms $\hat{E_n}$ of the equation  \eqref{deg1-general factorizable equation} satisfies $\hat{E_n} = \mathfrak{D}_n$ and hence $P(u_0;j)$ does not depend on $b_i$'s for $j \in \N$.

\item [(iii)] $u_0 \in  \left\{ - \dfrac{1}{a_1}, - \dfrac{2}{a_2}, \dots, - \dfrac{n}{a_n} \right \}$.

\end{itemize}
\end{lemma}

\begin{proof}
%One may easily see that the dominant terms $\hat{E_n}$ of the equation  \eqref{general factorizable equation} satisfies
%$\hat{E_n} \prec E_n$, where $\prec$ means that $\hat{E_n}$ is the summation of certain terms which come from those terms in $E_n$.
\begin{itemize}

\item [(i)] For a fixed $n \in \N$, we prove by contradiction. First rewrite the ODE \eqref{deg1-general factorizable equation} as
\begin{eqnarray} \label{deg1-1-general factorizable equation}
 \mathfrak{D}_n + \sum_{I \in \Omega'} c_{I} ' u^{i_0} (u')^{i_1} \cdots (u^{(n)})^{i_n}= 0,
\end{eqnarray}
then by Lemma \ref{weight-general factorizable equation}, one can see that $i_0 + 2i_1 + \cdots + (n + 1)i_n < n+ 1$ for any $I = (i_0, i_1, \dots, i_n ) \in \Omega'$.

Assume now $p \leq -2$, then for any term $u^{i_0} (u')^{i_1} \cdots (u^{(n)})^{i_n}$ in \eqref{deg1-1-general factorizable equation} with $ (i_0, i_1, \dots, i_n ) \not = (l, 0, \dots, 0 ), 0 \leq l \leq n+1$, according to Lemma \ref{weight-general factorizable equation}, its  order at $z = 0$ is
\begin{eqnarray*}
  &&  i_0 p + i_1 (p-1) + \dots + i_n (p - n)
\\
&=& \left( \sum_{k=0}^n i_k   \right) (p + 1) - \sum_{k=0}^n (k + 1) i_k
\\
&\geq& \left( \sum_{k=0}^n i_k   \right) (p + 1) - (n + 1)
\\
&>& \left( \sum_{k=0}^n (k + 1) i_k   \right) (p + 1) - (n + 1)
\\
&\geq&  p (n + 1) .
\end{eqnarray*}
 As the order of $(-1)^n \Pi_{k = 1}^n a_k u^{n+1}$ at $z = 0$ is $p (n + 1)$, which is lower than that of any other term in \eqref{deg1-1-general factorizable equation}, it cannot be balanced unless $u_0 = 0$. Consequently, we must have $p = -1$.

  %We prove by induction.
%
%\begin{itemize}
%\item [ b1)] It is obvious for $n = 1$.
%
%\item [ b2)] Suppose $p = -1$ for $n = K$, then for $n = K+1$, we prove by contradiction.
%Next, assume $p = p_0 \leq -2$. Then for $n = K$, there is only one term in $\hat{E_n}$ with the lowest pole order otherwise it can be balanced by other terms which contradicts to our assumption that $p = -1$ for $n = K$.
%\begin{eqnarray}
%\mathfrak{D}_{K+1} &=& [D - a_{k+1} u] \mathfrak{D}_K
%\\
%&=& \dfrac{d \mathfrak{D}_K}{d z} - a_{k+1} u  \mathfrak{D}_K
%\\
%&=& \sum_{(i_0, i_1, \dots, i_K )} c_{i_0, i_1, \dots, i_K} \left[ \sum_{k=0}^K i_k u^{i_0} (u')^{i_1} \cdots (u^{(K)})^{i_K}  \dfrac{u^{(k+1)}}{u^{(k)}} \right]
%\\
%&& - \sum_{(i_0, i_1, \dots, i_K )} a_{k+1} c_{i_0, i_1, \dots, i_K}    u^{i_0 + 1} (u')^{i_1} \cdots (u^{(K)})^{i_K}
%\\
%&=&  \sum_{(j_0, j_1, \dots, j_K, j_{K+1} )}   d_{j_0, j_1, \dots, j_K,  j_{K+1}  }    u^{j_0 } (u')^{j_1} \cdots (u^{(K)})^{j_K} (u^{(K + 1)})^{j_{K + 1 }} .
%\end{eqnarray}
%Here $(j_0, j_1, \dots, j_K, j_{K+1} ) = (i_0 + 1, i_1, \dots, i_K, 0 )$ or $(i_0, i_1, \dots, i_k - 1, i_{k+1} + 1, \dots )$, and in both cases we have $j_0 + 2 j_1 + \cdots +(K+1) j_K + (K+2) j_{K+1} = K+2$.
%
%
%
%\end{itemize}
\item [(ii)] As $p = - 1$,  we know that the order at $z = 0$  of each term  $$u^{i_0} (u')^{i_1} \cdots  (u^{(n)})^{i_n}$$ in \eqref{deg1-1-general factorizable equation} is no less than $-(n + 1)$.
     %$ i_0 p + i_1 (p-1) + \dots + i_n (p - n)  = -(n + 1)$.
      Therefore   $\hat{E_n}  $ consists of all terms with order $-(n + 1)$ at $z = 0$, and thus $\hat{E_n} = \mathfrak{D}_n$.

\item [(iii)]
%$u_0 \in \left \{ - \dfrac{1}{a_1}, - \dfrac{2}{a_2}, \dots, - \dfrac{n}{a_n} \right \}$.

To compute $u_0$, without loss of generality, we may assume $u(z) = \frac{u_0}{z}$.
 We then prove by induction. It is obvious for $n = 1$.
Suppose $u_0 \in \left \{ - \dfrac{1}{a_1}, - \dfrac{2}{a_2}, \dots, - \dfrac{n}{a_n} \right \}$ for an $n \in \mathbb{N}$ and we consider the $n+1$ case. If $u_0 = - \dfrac{k}{a_k}$ for some $1 \leq k \leq n$, then since $\hat{E}_n = \mathfrak{D}_n$ we have by direct checking that
\begin{eqnarray*}
  \mathfrak{D}_n \left(\frac{u_0}{z}\right) = 0, \quad \left(\dfrac{d \mathfrak{D}_n}{d z} \right)  \left(\frac{u_0}{z}\right) = 0,
% \dfrac{d \mathfrak{D}_K}{d z} - a_{k+1} u  \mathfrak{D}_K
\end{eqnarray*}
and hence
\begin{eqnarray*}
 \mathfrak{D}_{n+1} \left(\frac{u_0}{z}\right) = \left(\dfrac{d \mathfrak{D}_n}{d z} \right)  \left(\frac{u_0}{z}\right)
 - a_{k+1} \frac{u_0}{z}   \mathfrak{D}_n \left(\frac{u_0}{z}\right) = 0.
% \dfrac{d \mathfrak{D}_K}{d z} - a_{k+1} u  \mathfrak{D}_K
\end{eqnarray*}
For $u_0 = - \dfrac{n + 1}{a_{n + 1}}$, from \eqref{recursion1-general factorizable equation}   we know that

%{\footnotesize
\begin{eqnarray*}
&& \mathfrak{D}_{n+1} \left(\frac{u_0}{z}\right)
\\
&=& \left[ \sum_{(i_0, i_1, \dots, i_n )} c_{i_0, i_1, \dots, i_n} \left[ \sum_{k=0}^n i_k u^{i_0} (u')^{i_1} \cdots (u^{(n)})^{i_n}  \dfrac{u^{(k+1)}}{u^{(k)}} \right]  - a_{n+1} u  \mathfrak{D}_n \right] \left(\frac{u_0}{z}\right)
\\
&=& \left[ \sum_{(i_0, i_1, \dots, i_n )} \dfrac{ c_{i_0, i_1, \dots, i_n}}{z} \left[ - \sum_{k=0}^n  (k + 1) i_k u^{i_0} (u')^{i_1} \cdots (u^{(n)})^{i_n}  \right]  - a_{n+1} u  \mathfrak{D}_n \right] \left(\frac{u_0}{z}\right)
\\
&=& \left[ - \dfrac{1}{z} \sum_{k=0}^n  (k + 1) i_k \left[ \sum_{(i_0, i_1, \dots, i_n )} c_{i_0, i_1, \dots, i_n}   u^{i_0} (u')^{i_1} \cdots (u^{(n)})^{i_n}  \right]  - a_{n+1} u  \mathfrak{D}_n \right] \left(\frac{u_0}{z}\right)
\\
&=& \dfrac{1}{z} \left(- \sum_{k=0}^n  (k + 1) i_k \right) \mathfrak{D}_n \left(\frac{u_0}{z}\right) - a_{n+1} u  \mathfrak{D}_n  \left(\frac{u_0}{z}\right)
%\\
%&=& - \left(  K + 1 \right) \mathfrak{D}_K  - a_{K+1} u  \mathfrak{D}_K .
\\
&=& \left(  -  \dfrac{n + 1 }{z}    - a_{n+1} u \right) \mathfrak{D}_n \left(\frac{u_0}{z}\right)
\\
%&=& \left(  -  \dfrac{K + 1 }{z}    + a_{K+1}  \dfrac{K + 1}{z  a_{K + 1} }\right) \mathfrak{D}_K \left(\frac{u_0}{z}\right)
%\\
&=& 0.
\end{eqnarray*}
%}
%Therefore,
%\begin{eqnarray*}
%\mathfrak{D}_{K+1} \left(\frac{u_0}{z}\right)  &=& [  - \left(  K + 1 \right)    - a_{K+1} u ] \mathfrak{D}_K \left(\frac{u_0}{z}\right) .
%\end{eqnarray*}

On the other hand, it is easy to check that $ \dfrac{z^{n + 2}}{u_0} \mathfrak{D}_{n+1} \left(\frac{u_0}{z}\right)$ is a polynomial in $u_0$ of degree $n+1$ with coefficients depending only on $a_i, 1 \leq i \leq n+1$, and hence  the set of nonzero roots $u_0$ of $ z^{n + 2} \mathfrak{D}_{n+1} \left(\frac{u_0}{z}\right)=0 $ is $ \left \{ - \dfrac{1}{a_1}, - \dfrac{2}{a_2}, \dots, - \dfrac{n+1}{a_{n+1}} \right \}$.
\end{itemize}
\end{proof}

Now we denote by $R_n (u_0) = z^{n + 1}  \mathfrak{D}_{n} \left(\frac{u_0}{z}\right)$, then $R_n (u_0)$ is a polynomial of degree $n + 1$ in $u_0$ with the set of zeros $ \left \{0, - \dfrac{1}{a_1}, - \dfrac{2}{a_2}, \dots, - \dfrac{n}{a_{n}} \right \}$. Let the indicial equation of $\mathfrak{D}_{n} = 0$ be $P_n (u_0; j) = 0$. From Lemma \ref{order+lead coe-general factorizable equation}, we know that $P_n (u_0; j) = 0$ is also the indicial equation of \eqref{deg1-general factorizable equation} when all the $a_i$'s are non-zero. Let the indicial equation of $\dfrac{d \mathfrak{D}_n}{d z} = 0$ be $P_{n'} (u_0; j) = 0$, then we have
\begin{proposition} \label{relations on indicial euqations}
For any $n \in \N$,
%and  $ \mathfrak{D}_n = \mathfrak{D}_n (u(z)) = \sum_{I} c_{I} u^{i_0} (u')^{i_1} \cdots (u^{(n)})^{i_n}$
% it holds that
\begin{itemize}
\item [1)]
$
 P_{n'} (u_0; j)   = P_{n} (u_0; j) (j - n - 1) ;
$
\item [2)] $ P_{n + 1} (u_0; j) = P_{n} (u_0; j) (j - n - 1 - a_{n+1} u_0) - a_{n+1} R_n  (u_0 )$;

\item [3)]  If $u_0 = - \dfrac{k}{a_k}$, where $1 \leq k \leq n$, then
\begin{eqnarray*}
 P_{n + 1} (u_0; j) = 0 \Leftrightarrow  P_{n} (u_0; j) = 0 \; \text{or} \; j = n + 1 - k \dfrac{a_{n+1}}{a_k}.
\end{eqnarray*}
%where $1 \leq k \leq n$,
 If $u_0 = - \dfrac{n+1}{a_{n+1}}$,  then
\begin{eqnarray*}
 P_{n + 1} (u_0; j) = j  P_{n} (u_0; j)  - a_{n+1} R_n(u_0).
\end{eqnarray*}
\end{itemize}

\end{proposition}

\begin{remark}
If we choose $a_1 = a_2 = 1$, then $P_1(-1, j ) = j +1 $ and $ P_2(-1, j ) = j^2 - 1$.
\end{remark}

\begin{proof}
We set $v_k = u_0  (-1)^k k ! $ for $k = 0, 1, 2, \dots$, then we have $v_{k + 1} = - (k + 1) v_k$ and
from \eqref{indicial equation} and  $\hat{E}_n = \mathfrak{D}_n = \sum_{I \in \Omega} c_{I} u^{i_0} (u')^{i_1} \cdots (u^{(n)})^{i_n}$,
\begin{eqnarray*}
 P_n (u_0; j) = \sum_{I} c_{I} \left[ \sum_{k=1}^n i_k  \dfrac{\Pi_{\alpha=1}^n v_{\alpha}^{i_\alpha}}{v_k} (j - 1) (j - 2) \cdots (j-k) + i_0 v_0^{i_0-1} v_1^{i_1}   \cdots v_n^{i_n}  \right]
 %v_0^{i_0} v_1^{i_1} \cdots v_k^{i_k - 1} \cdots v_n^{i_n}
\end{eqnarray*}
Since   $\dfrac{d \mathfrak{D}_n}{d z} =\sum_{I} c_{I}
\left[ \sum_{k=0}^n i_k u^{i_0} (u')^{i_1} \cdots (u^{(n)})^{i_n}  \dfrac{u^{(k+1)}}{u^{(k)}} \right]   $, we have from \eqref{indicial equation},
%(v_{k+1})^{i_{k+1} + 1}
%{\color{blue}
%{\small
\begin{eqnarray*}
 && P_{n'} (u_0; j)
 \\
  &=& \sum_{I} c_{I}
 \\
 && \left[  \sum_{k=1}^n i_k \left(\sum_{\substack{m = 1 \\ m \not = k, k+1}}^{n} i_m v_0^{i_0} v_1^{i_1} \cdots v_k^{i_k - 1} v_{k+1}^{i_{k+1} + 1} \cdots v_m^{i_m - 1}\cdots v_n^{i_n}  (j - 1)  \cdots (j-m)\right. \right.
 \\
 && + v_0^{i_0 - 1} v_1^{i_1} \cdots v_k^{i_k - 1} v_{k+1}^{i_{k+1} + 1} \cdots   v_n^{i_n}
 \\
 &&+ (i_k - 1) v_0^{i_0} v_1^{i_1} \cdots v_k^{i_k - 2} v_{k+1}^{i_{k+1} + 1 }  \cdots v_n^{i_n}  (j - 1)  \cdots (j-k)
 \\
 &&   \left. + (i_{k + 1} + 1) v_0^{i_0} v_1^{i_1} \cdots v_k^{i_k -1} v_{k+1}^{i_{k+1} }  \cdots v_n^{i_n}  (j - 1)  \cdots (j-k) (j-k - 1) \right)
  \\
 &&   + i_0 \left( \sum_{m = 2}^{n} i_m v_0^{i_0 -1} v_1^{i_1 + 1}  \cdots v_m^{i_m - 1}\cdots v_n^{i_n}  (j - 1)  \cdots (j-m) \right.
  \\
 &&  \left.  \left. + (i_0 - 1) v_0^{i_0 -2} v_1^{i_1 + 1}  \cdots v_m^{i_m - 1}\cdots v_n^{i_n} + (i_1 + 1) v_0^{i_0 - 1} v_1^{i_1 }  \cdots v_n^{i_n} (j - 1) \right) \right].
\end{eqnarray*}
Now we compute each term in the above equality, for $1 \leq k \leq n$,
\begin{eqnarray*}
 && \sum_{I} c_{I} \Big[ \sum_{\substack{m = 1 \\ m \not = k, k+1}}^{n} i_m v_0^{i_0} v_1^{i_1} \cdots v_k^{i_k - 1} v_{k+1}^{i_{k+1} + 1} \cdots v_m^{i_m - 1}\cdots v_n^{i_n}  (j - 1)  \cdots (j-m)
 \\
 && + v_0^{i_0 - 1} v_1^{i_1}  \cdots v_k^{i_k - 1} v_{k+1}^{i_{k+1} + 1}  \cdots   v_n^{i_n}
 \\
 &&+ (i_k - 1) v_0^{i_0} v_1^{i_1} \cdots v_k^{i_k - 2} v_{k+1}^{i_{k+1} + 1 }  \cdots v_n^{i_n}  (j - 1)  \cdots (j-k)
 \\
 &&    + (i_{k + 1} + 1) v_0^{i_0} v_1^{i_1} \cdots v_k^{i_k -1} v_{k+1}^{i_{k+1} }  \cdots v_n^{i_n}  (j - 1)  \cdots (j-k) (j-k - 1) \Big ]
 \\
 &=& \sum_{I} c_{I} \Big[ - (k + 1)\sum_{\substack{m = 1 \\ m \not = k, k+1}}^{n} i_m v_0^{i_0} \cdots v_k^{i_k } v_{k+1}^{i_{k+1} } \cdots v_m^{i_m - 1}\cdots v_n^{i_n}  (j - 1)  \cdots (j-m)
 \\
 && - (k + 1) v_0^{i_0 - 1}  \cdots v_k^{i_k } v_{k+1}^{i_{k+1} } \cdots   v_n^{i_n}
 \\
 &&- (k + 1) (i_k - 1) v_0^{i_0}  \cdots v_k^{i_k - 1} v_{k+1}^{i_{k+1} }  \cdots v_n^{i_n}  (j - 1)  \cdots (j-k)
 \\
 &&    - (k + 1) (i_{k + 1} + 1) v_0^{i_0} \cdots v_k^{i_k } v_{k+1}^{i_{k+1} -1 }  \cdots v_n^{i_n}  (j - 1) \cdots (j-k) (j-k - 1) \Big]
 \\
 &=& - (k + 1) \left \{  P_n(u_0; j)   +   \sum_{I} c_{I} \Big[    - v_0^{i_0}  \cdots v_k^{i_k - 1} v_{k+1}^{i_{k+1} }  \cdots v_n^{i_n}  (j - 1)  \cdots (j-k) \right.
  \\
  && \left. + v_0^{i_0}  \cdots v_k^{i_k } v_{k+1}^{i_{k+1} -1 }  \cdots v_n^{i_n}  (j - 1)  \cdots (j-k) (j-k - 1) \Big] \right\}
  \\
  &=& - (k + 1)  \left \{  P_n(u_0; j) + \sum_{I} c_{I} \Big[    - v_0^{i_0} \cdots v_k^{i_k - 1} v_{k+1}^{i_{k+1} }  \cdots v_n^{i_n}  (j - 1)  \cdots (j-k) \right.
  \\
  &&  \left. - \dfrac{1}{k+1} v_0^{i_0}  \cdots v_k^{i_k -1 } v_{k+1}^{i_{k+1}  }  \cdots v_n^{i_n}  (j - 1) \cdots (j-k) (j-k - 1)\Big] \right\}
  \\
  &=& - (k + 1) [P_n(u_0; j) -   \sum_{I} c_{I}  \dfrac{j}{k+1}v_0^{i_0}  \cdots v_k^{i_k - 1} v_{k+1}^{i_{k+1} }  \cdots v_n^{i_n}  (j - 1)  \cdots (j-k)]
  \\
  &=& - (k + 1)  P_n(u_0; j) + \sum_{I} c_{I}   j v_0^{i_0} \cdots v_k^{i_k - 1} v_{k+1}^{i_{k+1} }  \cdots v_n^{i_n}  (j - 1)  \cdots (j-k).
\end{eqnarray*}
%}
Similarly, for $k = 0$,
\begin{eqnarray*}
 &&   \sum_{I} c_{I} \Big[ \sum_{m = 2}^{n} i_m v_0^{i_0 -1} v_1^{i_1 + 1}  \cdots v_m^{i_m - 1}\cdots v_n^{i_n}  (j - 1)  \cdots (j-m)
  \\
&& + (i_0 - 1) v_0^{i_0 -2} v_1^{i_1 + 1}  \cdots v_m^{i_m - 1}\cdots v_n^{i_n} + (i_1 + 1) v_0^{i_0 - 1} v_1^{i_1 }  \cdots v_n^{i_n} (j - 1) \Big]
 \\
&=&  -    P_n(u_0; j) +  \sum_{I} c_{I}  j v_0^{i_0 - 1} v_1^{i_1}  \cdots v_n^{i_n}   .
\end{eqnarray*}
Therefore
\begin{eqnarray*}
 P_{n'} (u_0; j) &=& - \sum_{k = 0}^n (k + 1)i_k  P_{n} (u_0; j) + j P_{n} (u_0; j)
 \\
 &=& P_{n} (u_0; j) (j - n - 1).
\end{eqnarray*}
As $
\mathfrak{D}_{n+1} = \dfrac{d \mathfrak{D}_n}{d z} - a_{n+1} u  \mathfrak{D}_n,
$
one can easily deduce that
\begin{eqnarray*}
 P_{n + 1} (u_0; j) &=&  P_{n'} (u_0; j) - a_{n+1} u_0 P_n  (u_0; j) - a_{n+1} R_n  (u_0 )
 \\
 &=& P_{n} (u_0; j) (j - n - 1 - a_{n+1} u_0) - a_{n+1} R_n  (u_0 ).
\end{eqnarray*}
Finally, 3) can be obtained by directly substituting the values of $u_0$ into the equality in 2).

\end{proof}

%\begin{remark}
%For example, $P_1(-1,j) = $ and $P_2(-1,j) = $.
%\end{remark}

\begin{remark}
The above method can  be used to get a similar relation between the indicial equations of   $H(y,y',\cdots,y^{(n)})=0$ and   $\dfrac{d H}{d z}=0$, where $y = y(z)$ and $H$ is a polynomial in $y$ and its derivatives.
\end{remark}

\noindent {\it Proof} of Theorem \ref{mero sol of general factorizable equation}.
Let $L_i = \{(a_1, a_2, \dots, a_n ) \in \C^n| a_i = 0 \}, i =1, 2 \dots, n$. Then for any $ \mathbf{a} \in \C^n \backslash \left( \cup_{i = 1}^n L_i \right)$, one can see immediately that the equation \eqref{deg1-general factorizable equation} has only one top degree term $ (-1)^n \Pi_{k = 1}^n a_k u^{n+1}$ and thus by Wiman-Valiron theory,  we conclude that it does not admit any transcendental entire solution. On the other hand, it is easy to see that the equation  \eqref{deg1-general factorizable equation} does not have any non-constant polynomial solution. Therefore any meromorphic solution of the ODE  \eqref{deg1-general factorizable equation} has at least one pole on $\C$. Then by Lemma \ref{order+lead coe-general factorizable equation}, any meromorphic solution  of \eqref{deg1-general factorizable equation} with a pole at $z = z_0 \in \C$ can be expressed as $u(z) =  \sum_{r = 0}^{\infty} u_r (z-z_0)^{r-1},u_0 \not = 0$.

Now  for any $1 \leq k \leq n, j   \in \N \cup \{0\}$, we define
\begin{eqnarray*}
H_{k, j} := \{ (a_1, a_2, \dots, a_n ) \in \C^n \backslash \left( \cup_{i = 1}^n L_i \right) | a_k^n P_{n} (u_0 ; j) = 0, u_0 = - \frac{k}{a_k}\},
\end{eqnarray*}
where  $P_{n} (u_0; j)$ is the indicial equation of \eqref{deg1-general factorizable equation} .
From Proposition \ref{relations on indicial euqations}-3), one can easily check by induction that $H_{k, j} \not \equiv \C^n \backslash \left( \cup_{i = 1}^n L_i \right)$ and thus it defines a hypersurface in $\C^n $.
Next, we define
\begin{eqnarray*}
S: = \left( \bigcup_{i = 1}^n L_i \right) \bigcup \left( { \bigcup_{\substack{1 \leq k \leq n,\\ j   \in \N \cup \{0\}}  } } H_{k, j}  \right),
\end{eqnarray*}
then according to the method used in \cite{Conte2010Ng,Eremenko2006}, for any $\mathbf{a} \in \C^n \backslash S$, since the equation \eqref{deg1-general factorizable equation}  does not have any  nonnegative integer Fuchs index for any $u_0 \in \{- \frac{k}{z_k} | k = 1,2, \dots, n\} $, all meromorphic solutions of the equation \eqref{deg1-general factorizable equation}  belong to the class $W$.

\section{Proof of Theorem \ref{solutions to 2-factorizble equations}}

We first recall some lemmas that will be needed.
\begin{lemma}\cite[p.~210]{Chuang1990Yang} \label{Growth estimate1-factorizabl ODEs}
Let $f$ and $g$ be two transcendental entire functions. Then
\begin{eqnarray*}
\lim_{r \rightarrow \infty} \dfrac{\log M(r, f \circ g)}{\log M(r, f)} = \infty, \quad \lim_{r \rightarrow \infty} \dfrac{T (r, f \circ g)}{  T(r, f)} = \infty, \\
\lim_{r \rightarrow \infty} \dfrac{\log M(r, f \circ g)}{\log M(r, g)} = \infty, \quad \lim_{r \rightarrow \infty} \dfrac{T (r, f \circ g)}{  T(r, g)} = \infty.
\end{eqnarray*}

\end{lemma}

\begin{lemma} \label{solution of Fisher equation with degree 2-factoriazble ODEs}
The equation
 \begin{eqnarray} \label{Fisher equation-factoriazble ODEs}
w''(z)  + c w'(z)-\frac{6}{\lambda } \left(w(z)-e_1\right) \left(w(z)-e_2\right)= 0, \lambda \not = 0
\end{eqnarray}
 has non-constant meromorphic solutions if and only if
 $$c  \left(c^2 \lambda +25 e_1-25 e_2\right) \left(c^2 \lambda -25 e_1+25 e_2\right) = 0$$ and they are given respectively as follows
\end{lemma}
\begin{itemize}
\item [1)] if $c = 0$, then the general solution to the equation \eqref{Fisher equation-factoriazble ODEs} is
\begin{equation*} \label{solution of Fisher equation with degree 2-1-factoriazble ODEs}
w_1(z) = \lambda  \wp \left(z - z_0;g_2,g_3\right) + \frac{1}{2} \left(  e_1+  e_2\right),
\end{equation*}
where $g_2 = \frac{3 \left(e_1-e_2\right){}^2}{\lambda ^2}$ and $z_0, g_3$ are arbitrary.

\item [2)] if $c^2 \lambda  = 25 (e_i - e_j) \not = 0, i, j \in \{1, 2\}$, then the general solution to the equation \eqref{Fisher equation-factoriazble ODEs} is
    \begin{eqnarray*} \label{solution of Fisher equation with degree 2-2-factoriazble ODEs}
w_2(z) &=& (e_i - e_j) e^{  \frac{- 2 c }{5} z} \wp\left( e^{ \frac{- c   }{5} z  } - \zeta_0; 0, g_3 \right) + e_j,
\end{eqnarray*}
where $\zeta_0, g_3 \in \C$ are arbitrary, see \cite{Ablowitz1979Zeppetella,Kudryashov2010}.
\end{itemize}

\begin{lemma} \cite[p.~53]{Laine1993} \label{solutions on general linear ODEs}
All solutions of the linear differential equation
\begin{eqnarray} \label{solutions to linear DEs}
L(f) := f^{(n)} + \alpha_{n-1}(z) f^{(n - 1)} + \cdots + \alpha_0(z) f  = 0
\end{eqnarray}
with entire coefficients $\alpha_0(z), \dots, \alpha_n(z)$, are entire functions.
\end{lemma}

\noindent{\it Proof} of Theorem \ref{solutions to 2-factorizble equations}.

\noindent Expanding \eqref{decomposition-2factors-1} gives
\begin{equation} \label{decomposition-2factors-2}
 u'' + u' \left(\alpha  a_1+\left(-2 a_1-a_2\right) u - b_1 - b_2\right)+ (u - \alpha ) \left(a_1 u+b_1\right) \left(a_2 u+b_2\right) = 0.
\end{equation}
%\begin{eqnarray} \label{decomposition-2factors-2}
%&& u'' + u' \left(\alpha  a_1+\left(-2 a_1-a_2\right) u - b_1 - b_2\right)+ a_1 a_2 u^3 + \nonumber
%\\
%&&u^2 \left(-\alpha  a_1 a_2+a_2 b_1+a_1 b_2\right) + u \left(-\alpha  a_2 b_1-\alpha  a_1 b_2+b_2 b_1\right)-\alpha  b_1 b_2 = 0.
%\end{eqnarray}
Before presenting the complete analysis of   meromorphic solutions of the ODE \eqref{decomposition-2factors-1}, we give a simple observation to derive some of its particular meromorphic solutions.  Let $G(z) = [D - a_1 u - b_1] (u - \alpha)$, then we have $[D - a_2 u  -b_2] G(z) = 0$ from which one can solve for $G(z) = \beta e^{\int a_2 u dz} e^{b_2 z}  , \beta \in \C$. If $\beta = 0$, from the Riccati equation $G(z) = [D - a_1 u - b_1] (u - \alpha) = 0$,
we are able to obtain the particular meromorphic solution  \eqref{decomposition-2factors-paricular solution} of the equation \eqref{decomposition-2factors-1}.

 For $\beta \not = 0$, in order to characterize all the meromorphic solutions of \eqref{decomposition-2factors-1}, we distinguish the following cases according to the values of $a_i $ and $ b_i, i = 1,2$.

%
%  \begin{center}
%\begin{tabular}{ l l  l   } % centered columns (4 columns)
% (1)  $  a_1  a_2 \not = 0,  2 a_1 + a_2 \not = 0 $, &   (2)  $  a_1  a_2 \not = 0,  2 a_1 + a_2   = 0$, & (3)  $ a_1   = 0,    a_2  \not = 0,  b_1 \not = 0, $
%\\
%   (4) $a_1 \not = 0, a_2   = 0,      b_2 \not = 0$, &
%(5)   other cases.   &
%\end{tabular}
%\end{center}

% \begin{center}
\begin{tabular}{ l l  l l   } % centered columns (4 columns)
 (I) & $  a_1  a_2 \not = 0,  2 a_1 + a_2 \not = 0 $, &   (II)  & $  a_1  a_2 \not = 0,  2 a_1 + a_2   = 0$,
\\
 (III) &  $ a_1   = 0,    a_2  \not = 0,  b_1 \not = 0, $ &   (IV) & $a_1 \not = 0, a_2   = 0,      b_2 \not = 0$,
\\
(V) &   other cases.   &
\end{tabular}
%\end{center}

%\begin{center}
%\begin{tabular}{ l l     } % centered columns (4 columns)
% (1)  $  a_1  a_2 \not = 0,  2 a_1 + a_2 \not = 0 $, &   (2)  $  a_1  a_2 \not = 0,  2 a_1 + a_2   = 0$,
%\\
%(3)  $ a_1   = 0,    a_2  \not = 0,  b_1 \not = 0 $,  &   (4) $a_1 \not = 0, a_2   = 0,      b_2 \not = 0$,
%\\
%(5)   other cases.   &
%\end{tabular}
%\end{center}
%\begin{eqnarray*}
%&(1) &  a_1  a_2 \not = 0,  2 a_1 + a_2 \not = 0,    (2) \quad  a_1  a_2 \not = 0,  2 a_1 + a_2   = 0,
%\\
%&(3)&   a_1   = 0,    a_2  b_1 \not = 0,   \qquad \; \; (4) \quad  a_2   = 0,    a_1  b_2 \not = 0,
%\\
%&(5)&  \text { other cases  }
% \end{eqnarray*}

% \begin{sidewaystable}[htpb]
%\newcommand{\tabincell}[2]{\begin{tabular}{@{}#1@{}}#2\end{tabular}}
%\caption{Meromorphic Solutions of ODE  \eqref{decomposition-2factors-1}.

%\label{table: 2nd order factorizable-1} % is used to refer this table in the text
%\end{sidewaystable}

%
% $$
% \begin{matrix}
%&(1) \;  a_1  a_2 \not = 0,  2 a_1 + a_2 \not = 0, & (2)\;  a_1  a_2 \not = 0,  2 a_1 + a_2   = 0,
%\\
%&(3) \;  a_1   = 0,    a_2  b_1 \not = 0, & (4) \;  a_2   = 0,    a_1  b_2 \not = 0,
%\\
%&(5) \;  other cases &
% \end{matrix}
% $$
%\begin{enumerate}

% case1
%\item
(I) $ a_1  a_2 \not = 0,  2 a_1 + a_2 \not = 0$.

For the convenience of applications, we first compare equation \eqref{decomposition-2factors-2} with the following second order ODE
\begin{equation} \label{decomposition-2factors-3}
\frac{1}{8} \left(k^2-d^2\right) \left(u-\alpha_1\right) \left(u-\alpha_2\right) \left(u-\alpha_3\right)+k (u-b) u'+u'' = 0 , k  \not = 0, k^2-d^2 \not = 0.
\end{equation}
One can see immediately that the ODE   \eqref{decomposition-2factors-2} is a special case of \eqref{decomposition-2factors-3} and further calculations imply that the ODE \eqref{decomposition-2factors-3} can be written in the form of \eqref{decomposition-2factors-1} if and only if the coefficients involved in \eqref{decomposition-2factors-3} satisfy either one of the following two conditions
\begin{eqnarray*} \label{decomposition-2factors-4}
%\prod \left[\left(a_i+a_j\right) (k \pm d)+2 a_n (k \mp d)-4 b k \right] = 0,
%\\
\prod \left[\left(\alpha_i+\alpha_j\right) (k + d)+2 \alpha_{k'} (k - d)-4 b k \right] = 0,
\\
\prod \left[\left(\alpha_i+\alpha_j\right) (k - d)+2 \alpha_{k'} (k + d)-4 b k \right] = 0,
\end{eqnarray*}
where $i, j , k' \in \{1, 2, 3\}$ are distinct  and the product is taken over all the permutations of $(1 2 3)$.
%Identifying the equations \eqref{decomposition-2factors-2} and \eqref{decomposition-2factors-3}   leads to the conditions
%\begin{eqnarray}
%-2 a_1-a_2-k=0, a_1 a_2+\frac{1}{8} \left(d^2-k^2\right)=0,
%\end{eqnarray}

We now come back to the ODE \eqref{decomposition-2factors-1}. Suppose $u(z)$ is a meromorphic solution of the ODE \eqref{decomposition-2factors-1} with a pole at $z = z_0$, W.L.O.G., we may assume $z_0 = 0$ then $u(z) = \sum_{j=p}^{+\infty} u_j z^{j},  - p \in \N, u_p \not = 0$.
Substituting the series expansion of $u(z)$ into \eqref{decomposition-2factors-1}  gives $p = - 1,u_{-1} =  - \frac{1}{a_1}$ or $- \frac{2}{a_2}$
and the Fuchs indices
%\begin{eqnarray*}
%\begin{cases}
%j = -1,  \frac{4 d}{d+k}, u_{-1} =  \frac{4}{k + d},
%\\
%j = -1, \frac{4 d}{d-k} , u_{-1} = \frac{4}{k - d}.
%\end{cases}
%\end{eqnarray*}
\begin{eqnarray*}
\begin{cases}
j = -1,  2- \frac{a_2}{a_1}, u_{-1} =  - \frac{1}{a_1},
\\
j = -1, 2- \frac{4 a_1}{a_2}, u_{-1} = - \frac{2}{a_2}.
\end{cases}
\end{eqnarray*}
We denote by $j_1 =   2- \frac{a_2}{a_1}, j_2 =   2- \frac{4 a_1}{a_2}$, then we have
\begin{eqnarray}  \label{decomposition-2factors-f.i.u}
j_1 = j_2 = 0, \text{or}\, \dfrac{2}{j_1} + \dfrac{2}{j_2} = 1.
\end{eqnarray}
For $G(z) \not \equiv 0$,  we let $H(z) = e^{\int a_2 u dz}$ which satisfies $H'(z) =  a_2  u(z) H(z)$ and
\begin{eqnarray} \label{decomposition-2factors-6}
[D - a_1 u - b_1] (u - \alpha) = \beta e^{b_2 z}   H(z), \beta \not = 0,
 \end{eqnarray}
hence, $u(z) $ is meromorphic if and only if $H(z)$ is meromorphic.
%Without loss of generality, we may assume $b_2 \not = 0$ otherwise we perform a translation in the independent variable $u(z)$.
By the substitution of $u = \dfrac{1}{a_2} \dfrac{H'}{H}$ into \eqref{decomposition-2factors-6}, we have
\begin{equation} \label{decomposition-2factors-7}
 -e^{b_2 z} a_2^2 \beta  H^3  +\alpha  a_2^2 b_1 H^2+a_2 H H''+ (\alpha  a_1 a_2 -a_2 b_1 ) H H'- (a_1 + a_2) \left(H'\right)^2 =0.
\end{equation}

If we let $H(z) = e^{-b_2 z} h(z)$, which implies $u(z) = \frac{h'(z)}{a_2 h(z)}-\frac{b_2}{a_2}$ and $u(z) $ is meromorphic if and only if so is $h(z)$, then the ODE \eqref{decomposition-2factors-7} reduces to
\begin{eqnarray} \label{decomposition-2factors-8}
&&h^2 \left(a_2 b_1-a_1 b_2\right) \left(\alpha  a_2+b_2\right)+h h' \left(a_1 \left(\alpha  a_2+2 b_2\right)-a_2 b_1\right)-a_2^2 \beta h^3+a_2 h h''
\\
&&- \left(a_1 + a_2\right) \left(h'\right)^2=0, \beta \not = 0. \nonumber
\end{eqnarray}
It is obvious that the ODE \eqref{decomposition-2factors-8} does not have any polynomial solutions. By Wiman-Valiron theory \cite[p.~51]{Laine1993}, we can conclude that \eqref{decomposition-2factors-8} does not have any transcendental entire solutions. Hence, each meromorphic solution of the ODE \eqref{decomposition-2factors-8} should have at least one pole on $\C$. Next  suppose $h(z)$ is a meromorphic solution of \eqref{decomposition-2factors-8}, W.L.O.G, we assume that it has  a pole at $z = 0$ and $h(z) = \sum_{j=p}^{+\infty} h_j z^{j}, - p \in \N, h_p \not = 0$.
Now the following cases are distinguished.
\begin{itemize}
% case1-1

 %   {\color{red} Find the solutions for a particular case}
    %, say $a_2 = -3 a_1$, which implies $i_1 = 5, i_2 = \dfrac{10}{3}$, to investigate their pole structures.}

% case1-2
\item [(A)] If both of $j_1$ and $j_2$ are integers, then by solving the Diophantine equation  \eqref{decomposition-2factors-f.i.u}, we have three  choices
\begin{eqnarray*}
\begin{cases}
j_1 = j_2 = 0 \Leftrightarrow a_2 = 2 a_1,
\\
\{j_1,  j_2\} = \{3, 6\} \Leftrightarrow a_2 = - a_1 \, \text{or} \, a_2 =  -4 a_1,
\\
\{j_1,  j_2\} = \{1, -2\} \Leftrightarrow a_2 =  a_1 \, \text{or} \, a_2 =  4 a_1.
\end{cases}
\end{eqnarray*}

%\begin{itemize}

% case1-2-1
%\item [1b-1]
{\it Subcase A0}. If  $a_2 = 2 a_1$, then the ODE \eqref{decomposition-2factors-8} reduces to
\begin{eqnarray} \label{decomposition-2factors-9}
\left(2 b_1-b_2\right) h^2 \left(2 \alpha  a_1+b_2\right)+2 h h' \left(\alpha  a_1-b_1+b_2\right)
\\
-4 a_1 \beta  h^3+2 h h'' -3 \left(h'\right)^2 = 0, \quad \beta \not = 0. \nonumber
\end{eqnarray}
One can see that there does not exist any negative integer $p $ with $ h_p \not = 0$ such that $h(z) = \sum_{j=p}^{+\infty} h_j z^{j} $ satisfies \eqref{decomposition-2factors-9} and therefore in this case all the meromorphic solutions of \eqref{decomposition-2factors-1} are those given in \eqref{decomposition-2factors-paricular solution}.

% case1-2-2
{\it Subcase A1}. For $a_2 = a_1$, the ODE \eqref{decomposition-2factors-8} reduces to
\begin{eqnarray} \label{decomposition-2factors-10}
&&h(z) \left(\alpha  a_1-b_1+2 b_2\right) h'(z)-\left(b_2-b_1\right) h(z)^2 \left(\alpha  a_1+b_2\right) - a_1 \beta  h(z)^3
\\
&&+h(z) h''(z)-2 h'(z)^2 = 0. \nonumber
\end{eqnarray}
Let $h(z) = \dfrac{1}{v(z)}$, then the ODE \eqref{decomposition-2factors-10} reduces to a linear ODE
\begin{equation*} \label{decomposition-2factors-11}
-v''  +\left(-\alpha  a_1+b_1-2 b_2\right) v'+v \left(b_1-b_2\right) \left(\alpha  a_1+b_2\right)  - a_1 \beta = 0,
\end{equation*}
with solutions
\begin{eqnarray*}
v(z) =
\begin{cases}
c_2 e^{  \left(-\alpha  a_1-b_2\right) z}+c_1 e^{\left(b_1-b_2\right) z} + \frac{a_1 \beta }{\left(b_1-b_2\right) \left(\alpha  a_1+b_2\right)}, \prod_{i=1}^2  \left(\alpha  a_1+b_i\right) \left(b_1-b_2\right) \not = 0 %\left(b_1 + \alpha a_1\right) \left(b_1-b_2\right) \left(\alpha  a_1+b_2\right) \not = 0;
\\
\frac{\beta  b_1}{\alpha  \left(b_1-b_2\right){}^2}+e^{\left(b_1-b_2\right) z} \left(c_2 z+c_1\right), b_1 + \alpha a_1 = 0, \left(b_1-b_2\right) \left(\alpha  a_1+b_2\right) \not = 0;
\\
c_2-\frac{c_1 e^{-  \left(\alpha  a_1+b_2\right) z}+a_1 \beta  z}{\alpha  a_1+b_2}, b_1 = b_2, \alpha  a_1+b_2 \not = 0;
\\
\frac{c_1 e^{z \left(\alpha  a_1+b_1\right)}+a_1 \left(\alpha  c_2+\beta  z\right)+b_1 c_2}{\alpha  a_1+b_1}, b_1 \not = b_2, \alpha  a_1+b_2 = 0;
\\
-\frac{1}{2} a_1 \beta  z^2+c_2 z+c_1, b_1 = b_2 = - \alpha  a_1.
\end{cases}
\end{eqnarray*}
%{\color{red} change the classification of the coefficients in the above expression.}

After substitution, we obtain the follow meromorphic solutions of the ODE \eqref{decomposition-2factors-1}
\begin{eqnarray*}
u(z) =
\begin{cases}
\frac{\left(b_1-b_2\right) \left(\alpha  a_1+b_2\right) \left(\alpha  a_1 c_2-b_1 c_1 e^{z \left(\alpha  a_1+b_1\right)}\right)-a_1 \beta  b_2 e^{z \left(\alpha  a_1+b_2\right)}}{a_1 \left(a_1 \left(\alpha  \left(b_1-b_2\right) \left(c_1 e^{z \left(\alpha  a_1+b_1\right)}+c_2\right)+\beta  e^{z \left(\alpha  a_1+b_2\right)}\right)+\left(b_1-b_2\right) b_2 \left(c_1 e^{z \left(\alpha  a_1+b_1\right)}+c_2\right)\right)},
\\
\left(b_1 + \alpha a_1\right) \left(b_1-b_2\right) \left(\alpha  a_1+b_2\right) \not = 0;
\\
\frac{\alpha  \left(\alpha  \left(b_1-b_2\right){}^2 e^{b_1 z} \left(b_1 \left(c_2 z+c_1\right)+c_2\right)+\beta  b_1 b_2 e^{b_2 z}\right)}{b_1 \left(\alpha  \left(b_1-b_2\right){}^2 e^{b_1 z} \left(c_2 z+c_1\right)+\beta  b_1 e^{b_2 z}\right)},
\\
b_1 + \alpha a_1 = 0, \left(b_1-b_2\right) \left(\alpha  a_1+b_2\right) \not = 0;
\\
\frac{e^{z \left(\alpha  a_1+b_2\right)} \left(b_2^2 c_2-a_1 \left(\beta +b_2 \left(\beta  z-\alpha  c_2\right)\right)\right)+\alpha  a_1 c_1}{a_1 \left(e^{z \left(\alpha  a_1+b_2\right)} \left(a_1 \left(\beta  z-\alpha  c_2\right)-b_2 c_2\right)+c_1\right)}, b_1 = b_2, \alpha  a_1+b_2 \not = 0;
\\
\frac{a_1 \left(\alpha  a_1 \left(\alpha  c_2+\beta  z\right)-\beta +\alpha  b_1 c_2\right)-b_1 c_1 e^{z \left(\alpha  a_1+b_1\right)}}{a_1 \left(c_1 e^{z \left(\alpha  a_1+b_1\right)}+a_1 \left(\alpha  c_2+\beta  z\right)+b_1 c_2\right)}, b_1 \not = b_2, \alpha  a_1+b_2 = 0;
\\
\frac{-2 a_1 \left(\alpha  c_1+z \left(\beta +\alpha  c_2\right)\right)+\alpha  a_1^2 \beta  z^2+2 c_2}{a_1 \left(a_1 \beta  z^2-2 \left(c_2 z+c_1\right)\right)}, b_1 = b_2 = - \alpha a_1,
\end{cases}
\end{eqnarray*}
where $\beta, c_1, c_2 \in \C$ are arbitrary.

%{\color{red}
%\begin{remark}
%For $a_1 = a_2$, the solutions have been checked by Mathematica!
%\end{remark}
%}

\begin{remark}
In general, the above $u(z)$ does not belong to the class $W$.
%For $\beta = 0$ and $c_2 a_1 = -1$, the above solution with $\left(b_1 + \alpha a_1\right) \left(b_1-b_2\right) \left(\alpha  a_1+b_2\right) \not = 0$ reduces to \eqref{decomposition-2factors-paricular solution}.
 For $ b_1 = b_2 = \alpha = 0, \beta = 1, a_1 = a_2 = -1$, we recover the general solution $u(z) = \dfrac{1}{z - a} + \dfrac{1}{z - b}, a + b = -2 c_2, ab= 2c_1, c_1, c_2 \in \C$ to the ODE $u'' +3 u u ' + u^3 = 0 $ \cite{Painleve1900,Conte1993FordyPickering}.
\end{remark}

\begin{remark}
It seems that we have three arbitrary constants $\beta, c_1$ and $c_2$ to a second order ODE, but actually the arbitrariness of $\beta$ can be absorbed into $c_1$ and $c_2$.
\end{remark}

%case1-2-3
{\it Subcase A2}. For the case $a_2 = - a_1$, the ODE \eqref{decomposition-2factors-8} reduces to the Fisher equation
\begin{eqnarray} \label{decomposition-2factors-12}
&&h''(z) -\left(-\alpha  a_1+b_1+2 b_2\right) h'(z)-\left(b_1+b_2\right) h(z) \left(\alpha  a_1-b_2\right)
\\
&&+a_1 \beta  h(z)^2 = 0. \nonumber
\end{eqnarray}
According to Lemma \ref{solution of Fisher equation with degree 2-factoriazble ODEs}, the necessary condition for the existence of meromorphic solutions of the ODE \eqref{decomposition-2factors-12} is
\begin{equation*} \label{decomposition-2factors-13}
c  \left(c^2 \lambda +25 e_1-25 e_2\right) \left(c^2 \lambda -25 e_1+25 e_2\right) = 0,
\end{equation*}
where
\begin{equation*} \label{decomposition-2factors-14}
\begin{cases}
c = \alpha  a_1 - b_1 - 2 b_2,\\
\lambda = - \dfrac{6}{a_1 \beta },\\
e_1 = 0, e_2 = \dfrac{\left(b_1+b_2\right)  \left(\alpha  a_1-b_2\right)}{a_1 \beta }.
\end{cases}
\end{equation*}

\begin{itemize}
\item [(i)] If $c = 0$, then the general solution to the equation \eqref{decomposition-2factors-12} is
\begin{equation*}
h(z) = \lambda  \wp \left(z - z_0;g_2,g_3\right) + \frac{1}{2} \left(  e_1+  e_2\right),
\end{equation*}
where $g_2 = \frac{3 \left(e_1-e_2\right){}^2}{\lambda ^2}$ and $z_0, g_3$ are arbitrary.

\item [(ii)] For $c^2 \lambda  = 25 (e_i - e_j) \not = 0, i, j \in \{1, 2\}$, then the general solution to the equation \eqref{decomposition-2factors-12} is
\begin{eqnarray*}
h(z) &=& (e_i - e_j) e^{  \frac{- 2 c }{5} z} \wp\left( e^{ \frac{- c   }{5} z  } - \zeta_0; 0, g_3 \right) + e_j,
\end{eqnarray*}

\end{itemize}

After substitution, we obtain the following meromorphic solutions of the ODE \eqref{decomposition-2factors-1}
\begin{itemize}
\item [(i)] if $a_2 = -a_1, c = 0 $,
\begin{equation*}
u(z) =  \dfrac{12 \wp '\left(z - z_0;g_2,g_3\right)}{a_1 [\left(b_2-\alpha  a_1\right){}^2-12 \wp \left(z - z_0;g_2,g_3\right)]} + \dfrac{b_2}{a_1},
\end{equation*}
where $g_2 = \dfrac{1}{12} \left(b_2-\alpha  a_1\right){}^4$ and $z_0, g_3$ are arbitrary.

\item [(ii)] for $ a_2 = -a_1, c^2 \lambda  = 25 (e_i - e_j) \not = 0, i, j \in \{1, 2\}$,
\begin{eqnarray*}
u(z) = \dfrac{b_2}{a_1}+\frac{c \left(e_i-e_j\right) \left(2 e^{\frac{c z}{5}} \wp \left(e^{-\frac{1}{5} (c z)} - \zeta_0;0,g_3\right)+\wp '\left(e^{-\frac{1}{5} (c z)} - \zeta_0;0,g_3\right)\right)}{a_1 \left[ 5 \left(e_i-e_j\right) e^{\frac{c z}{5}} \wp \left(e^{-\frac{1}{5} (c z)} - \zeta_0;0,g_3\right)+5 e_j e^{\frac{3 c z}{5}}\right]},
\end{eqnarray*}
where $ \zeta_0, g_3$ are arbitrary.

\end{itemize}

\begin{remark}
The above solutions may degenerate to rational functions in exponential or rational functions due to the degeneration of $\wp$.
\end{remark}

%{\color{red}
%\begin{remark}
%For $a_1 = - a_2$, the solutions have been checked by Mathematica!
%\end{remark}
%}

%case1-2-4
{\it Subcase A3}.  If  $a_2 = - 4 a_1$, set $u(z) = w(z) + \alpha$, then the ODE \eqref{decomposition-2factors-1} becomes
\begin{eqnarray}  \label{decomposition-2factors-Case14-1}
[D + 4 a_1 w - b'_2] [D - a_1 w - b'_1] w = 0,
\end{eqnarray}
where $b'_1 = b_1 +  \alpha a_1, b'_2 = b_2 - 4  \alpha a_1$. The compatibility conditions for the existence of meromorphic solutions of \eqref{decomposition-2factors-Case14-1} are $b'_2 = - 2b'_1, 2 b'_1, $ or $ -6 b'_1$.

Note that, under the compatibility conditions, the equation \eqref{decomposition-2factors-Case14-1}   can be factorized into another form
\begin{eqnarray}  \label{decomposition-2factors-Case14-2}
[D + \alpha_1 w - \beta_2] [D - \alpha_1 w - \beta_1] (w - \alpha_0) = 0.
\end{eqnarray}
Here
\begin{eqnarray*}  \label{decomposition-2factors-Case14-3}
\left(
\begin{array}{c}
 \alpha_1 \\
 \beta_1 \\
 \beta_2 \\
 \alpha _0 \\
\end{array}
\right)=\left(
\begin{array}{c}
 -2 a_1 \\
 -2 b'_1 \\
 b'_1 \\
 0 \\
\end{array}
\right), \left(
\begin{array}{c}
 -2 a_1 \\
 b'_1 \\
 2 b'_1 \\
 0 \\
\end{array}
\right)
\text{or}
\left(
\begin{array}{c}
 -2 a_1 \\
 -3 b'_1 \\
 0 \\
 -\frac{b'_1}{a_1} \\
\end{array}
\right)
\end{eqnarray*}
which correspond to $b'_2 = - 2b'_1, 2 b'_1, $ or $ -6 b'_1$ respectively.
\begin{remark}
The ODE \eqref{decomposition-2factors-1} admits more than one distinct factorizations only for some specific cases including \eqref{decomposition-2factors-Case14-1} with the compatibility conditions satisfied.
\end{remark}

Since the equation \eqref{decomposition-2factors-Case14-2} shares the same form as \eqref{decomposition-2factors-1} with $a_2 = - a_1$, one can obtain its meromorphic solutions given below by using the results of  {\it Subcase A2}.

For $b'_2 = - 2b'_1$,
\begin{eqnarray*}  \label{decomposition-2factors-Case14-4}
u(z) = -\frac{12 \wp '\left(z - z_0;g_2,g_3\right)}{2 a_1 \left(\left(\alpha  a_1+b_1\right){}^2-12 \wp \left(z- z_0;g_2,g_3\right)\right)}-\frac{b_1-\alpha  a_1}{2 a_1},
\end{eqnarray*}
where $g_2 = \dfrac{1}{12} \left(b_1 + \alpha  a_1\right){}^4$ and $z_0, g_3$ are arbitrary.

For $b'_2 =  2b'_1 \not = 0$,
\begin{eqnarray*}  \label{decomposition-2factors-Case14-5}
u(z) &=& -\frac{c \left(e_2-e_1\right) \left[2 e^{\frac{c z}{5}} \wp \left(e^{-\frac{1}{5} (c z) } - \zeta_0;0,g_3\right)+\wp '\left(e^{-\frac{1}{5} (c z)} - \zeta_0;0,g_3\right)\right]}{2 a_1 \left[5 \left(e_2-e_1\right) e^{\frac{c z}{5}} \wp \left(e^{-\frac{1}{5} (c z)} - \zeta_0;0,g_3\right)+5 e_1 e^{\frac{3 c z}{5}}\right]}
\\
&&-\frac{b_1}{a_1},
\end{eqnarray*}
where $c = -5 \left(\alpha  a_1+b_1\right) \not = 0, e_1 = 0$, $e_2 = \dfrac{3\left(\alpha  a_1+b_1\right)^2}{a_1 \beta}, $ and $ \beta \not = 0, \zeta_0, g_3$ are arbitrary.

For $b'_2 =  - 6 b'_1 \not = 0$,
\begin{eqnarray*}  \label{decomposition-2factors-Case14-5}
u(z) = \alpha -\frac{c \left(e_2-e_1\right) \left[2 e^{\frac{c z}{5}} \wp \left(e^{-\frac{1}{5} (c z)} - \zeta_0;0,g_3\right)+\wp '\left(e^{-\frac{1}{5} (c z)} - \zeta_0;0,g_3\right)\right]}{2 a_1 \left[5 \left(e_2-e_1\right) e^{\frac{c z}{5}} \wp \left(e^{-\frac{1}{5} (c z)} - \zeta_0;0,g_3\right)+5 e_1 e^{\frac{3 c z}{5}}\right]},
\end{eqnarray*}
where $c =  5 \left(\alpha  a_1+b_1\right) \not = 0, e_1 = 0$, $e_2 = \dfrac{3\left(\alpha  a_1+b_1\right)^2}{a_1 \beta}, $ and $ \beta \not = 0,  \zeta_0, g_3$ are arbitrary.

\begin{remark}
The above solutions may degenerate to rational functions in exponential or rational functions due to the degeneration of $\wp$.
%If $b'_2 = b'_1 = 0$, the solutions are given in \eqref{decomposition-2factors-Case14-4}.
\end{remark}

%{\color{red}
%\begin{remark}
%For $ a_2 = - 4 a_1 $, the solutions have been checked by Mathematica!
%\end{remark}
%}

%case1-2-5
{\it Subcase A4}. If  $a_2 = 4 a_1$, then the ODE \eqref{decomposition-2factors-8} reduces to
\begin{eqnarray*}  \label{decomposition-2factors-Case15-1}
&&2 h(z) \left(2 \alpha  a_1-2 b_1+b_2\right) h'(z)+\left(4 b_1-b_2\right) h(z)^2 \left(4 \alpha  a_1+b_2\right)-16 a_1 \beta  h(z)^3 \nonumber
\\
&&+4 h(z) h''(z)-5 h'(z)^2 = 0,
\end{eqnarray*}
with Fuchs indices $j = -1, 1$ and the compatibility condition
\begin{eqnarray} \label{decomposition-2factors-Case15-11}
 2 \alpha  a_1-2 b_1+b_2 = 0.
\end{eqnarray}
Again, we make the change of variables $u \mapsto u + \alpha, b_1 \mapsto b'_1, b_2 \mapsto b'_2$, where $b'_1 = b_1 +  \alpha a_1, b'_2 = b_2 + 4  \alpha a_1$, then with the compatibility condition $b'_2 = 2  b'_1$ satisfied,
the ODE \eqref{decomposition-2factors-7} reduces to
\begin{equation*} \label{decomposition-2factors-Case15-2}
  - 16 a_1 \beta    e^{2 b'_1 z} H(z)^3- 4 b'_1 H(z) H'(z)+ 4 H(z) H''(z)
 - 5 H'(z)^2 = 0, \quad  \beta \not = 0. %\nonumber
\end{equation*}
For $b'_1 \not = 0$, performing the transformation $H (z) = v(\zeta), \zeta = e^{b'_1 z } $ gives
%\begin{eqnarray} \label{reduction2 of generalized fisher}
%2 q_2^2 v\left(e^{-\frac{q_2 z}{\lambda }}\right) v''\left(e^{-\frac{q_2 z}{\lambda }}\right)-q_2^2 v'\left(e^{-\frac{q_2 z}{\lambda }}\right){}^2-4 \beta  \lambda  v\left(e^{-\frac{q_2 z}{\lambda }}\right){}^3
%\end{eqnarray}
\begin{eqnarray} \label{decomposition-2factors-Case15-3}
\frac{16 a_1 \beta }{(b'_1)^{2}} v^3- 4 v v''+ 5 v'^2 = 0, \, \, ' = \dfrac{d}{d \zeta}.
\end{eqnarray}

%Multiplying $ v'$ on both sides of \eqref{decomposition-2factors-Case15-3}, it becomes
%\begin{eqnarray*} \label{decomposition-2factors-Case15-4}
%4 v v'  \left(v'' - \frac{24 a_1 \beta }{(b'_1)^2} v^2 \right) - 10  v' \left(\dfrac{(v')^2}{2}  - \frac{8 a_1 \beta }{(b'_1)^2} v^3 \right) = 0,
%\end{eqnarray*}
%which can be written as
%\begin{eqnarray} \label{decomposition-2factors-Case15-5}
%4 v  \varphi'(\zeta) +   10v' \varphi(\zeta) = 0,
%\end{eqnarray}
%where $\varphi(\zeta) = \dfrac{(v')^2}{2} - \dfrac{8 a_1 \beta }{b_1^2} v^3.$
%Integrate the ODE \eqref{decomposition-2factors-Case15-5},

Upon integration of \eqref{decomposition-2factors-Case15-3}, we have
\begin{eqnarray*} \label{decomposition-2factors-Case15-6}
c_0 v^5 + \left((v')^2 - \dfrac{16 a_1 \beta }{b_1^2} v^3\right)^2  = 0, c_0 \, \text{arbitrary},
\end{eqnarray*}
which has the general solution
\begin{eqnarray*} \label{decomposition-2factors-Case15-7}
v(\zeta) =
\begin{cases}
\frac{\left( b_1'\right){}^2}{4 a_1 \beta  \left(\zeta + c_1 \right){}^2}, c_0 = 0,
\\
-\frac{256 c_0 ( b_1')^4}{\left(256 a_1 \beta +c_0 [ b_1'\left(\zeta-c_1\right)]{}^2\right){}^2}, c_0 \not = 0,
\end{cases}
c_0, c_1 \, \text{arbitrary}.
\end{eqnarray*}
%where $g_2 = \dfrac{C \beta  \lambda}{2 q_2^2},g_3 = 0, C, \zeta_0 \in \C$.

For $b'_2 = 2 b'_1 = 0$, the ODE \eqref{decomposition-2factors-1} reduces to
\begin{eqnarray}  \label{decomposition-2factors-Case15-7-1}
[D - 4 a_1 u ] [D - a_1 u] u = 0,
\end{eqnarray}
which admits another factorization
\begin{eqnarray*}  \label{decomposition-2factors-Case15-7-2}
[D - \alpha_1 u ] [D - \alpha_1 u] u = 0, \alpha_1 = 2 a_1.
\end{eqnarray*}
The above ODE belongs to {\it Subcase A1}, and hence we can obtain the  following solution of the ODE \eqref{decomposition-2factors-Case15-7-1}
\begin{eqnarray*}
u(z) = -\frac{1}{2 a_1 (z-c_0 )}-\frac{1}{2 a_1 (z-   c_1)}, \, c_0, c_1 \, \text{arbitrary}.
\end{eqnarray*}

After substitution, with \eqref{decomposition-2factors-Case15-11} satisfied, we obtain the meromorphic solutions of the ODE \eqref{decomposition-2factors-1}
\begin{equation*} \label{decomposition-2factors-Case15-8}
u(z) =
\begin{cases}
\alpha -\frac{1}{2 a_1 (z-c_0 )}-\frac{1}{2 a_1 (z-   c_1)}, \alpha a_1 + b_1 = 0,
\\
\alpha -\frac{\left(\alpha  a_1+b_1\right) e^{z \left(\alpha  a_1+b_1\right)}}{2 a_1 \left(e^{z \left(\alpha  a_1+b_1\right)}+c_1\right)}, c_0 = 0, \alpha a_1 + b_1 \not = 0,
\\
\alpha -\frac{c_0 \left(\alpha  a_1+b_1\right){}^3 e^{z \left(\alpha  a_1+b_1\right)} \left(e^{z \left(\alpha  a_1+b_1\right)}-c_1\right)}{a_1 \left(256 a_1 \beta +c_0 \left(\alpha  a_1+b_1\right){}^2 \left(e^{z \left(\alpha  a_1+b_1\right)}-c_1\right){}^2\right)}, c_0 \not = 0, \alpha a_1 + b_1 \not = 0,
\end{cases}
\end{equation*}
where $c_0, c_1$ are arbitrary.

%{\color{red}
%\begin{remark}
%For $ a_2 = 4a_1 \not  = 0$, the solutions have been checked by Mathematica!
%\end{remark}
%}
Again, it seems that we have three arbitrary constants to a second order ODE, but actually the arbitrariness of $\beta$ can be absorbed into $c_1$ and $c_2$.

% case1-3
\item [(B)] If only one of $j_1$ and $j_2$ is an integer, then we should have $a_2 \not = \pm a_1, \pm 4 a_1  $ or $ 2 a_1$.
For the ODE \eqref{decomposition-2factors-8},
% Suppose $h(z)$ is a meromorphic solution of \eqref{decomposition-2factors-8}, W.L.O.G, we assume that it has  a pole at $z = 0$ and
% $h(z) = \sum_{j=p}^{+\infty} h_j z^{j}, p \in \Z^{-}, h_p \not = 0$  then
 one can check that $p = -2, h_{-1} = \frac{2 \left(a_2-2 a_1\right)}{a_2^2 \beta }$ and its Fuchs indices  are $j = - 1,  2-\frac{4 a_1}{a_2} $.

{\it Subcase B1}. If $2-\frac{4a_1}{a_2}   \in \N \cup \{0\}$   and $2-\frac{a_2}{a_1 }  \not \in \Z$, i.e.,
%\N \cup \{0\}
 $2 - \dfrac{4a_1}{a_2}
\not \in \N \backslash \{1, 2, 3, 4, 6\}$,
% $a_2 \not =\dfrac{4 a_1}{ 2 - n}$ for any $n \in \N \backslash \{ 1, 3, 4, 6\}$,
  then the method in \cite{Conte2010Ng,Eremenko2006} fails and so in this case meromorphic solutions are unfortunately {\it not  yet known}.

{\it Subcase B2}. If $2-\frac{4a_1}{a_2} \not \in \N \cup \{0, -2\}$ ($a_2 \not = a_1$), then $h$ belongs to the class $W$ and so is $u$.
%and $2-\frac{a_2}{a_1 }   \in \N \cup \{0\}$
% and $2-\frac{a_2}{a_1}   \in \N \cup \{0\}$

%some compatibility conditions should be satisfied in order that meromorphic solutions exist.

Using the argument used in \cite{Conte2010Ng,Eremenko2006}, one can find all the meromorphic solutions of \eqref{decomposition-2factors-8}
\begin{equation*}  \label{decomposition-2factors-Case1-3-2-1}
h(z) =
\begin{cases}
\frac{2 \left(a_2-2 a_1\right) \left(\alpha  a_1+b_1\right){}^2 e^{2 \left(z-z_0\right) \left(\alpha  a_1+b_1\right)}}{a_2^2 \beta  \left(e^{\left(z-z_0\right) \left(\alpha  a_1+b_1\right)}-1\right){}^2},  b_2 =  2 \alpha  a_1-\alpha  a_2+2 b_1, \alpha  a_1+b_1 \not = 0;
\\
\frac{2 \left(a_2-2 a_1\right) \left(\alpha  a_1+b_1\right){}^2}{a_2^2 \beta  \left(e^{\left(z-z_0\right) \left(\alpha  a_1+b_1\right)}-1\right){}^2}, b_2 = \frac{-2 \alpha  a_1^2-2 a_1 b_1+a_2 b_1}{a_1},  \alpha  a_1+b_1 \not = 0;
\\
-\frac{\left(2 a_1-a_2\right) \left(\alpha  a_1+b_1\right){}^2}{2 a_1^2 \beta  \left(e^{\frac{a_2 \left(z-z_0\right) \left(\alpha  a_1+b_1\right)}{4 a_1}}-e^{-\frac{a_2 \left(z-z_0\right) \left(\alpha  a_1+b_1\right)}{4 a_1}}\right){}^2},  b_2 = \frac{a_2 b_1-\alpha  a_1 a_2}{2 a_1},  \alpha  a_1+b_1 \not = 0;
\\
\frac{2 \left(a_2-2 a_1\right)}{a_2^2 \beta  \left(z-z_0\right){}^2}, b_1 = -\alpha  a_1, b_2 = - \alpha  a_2  .
\end{cases}
\end{equation*}

After substitution,  we obtain the meromorphic solutions of the ODE \eqref{decomposition-2factors-1} in {\it Subcase B2}
%which includes case (a)
\begin{eqnarray*} \label{decomposition-2factors-Case14-8}
u(z) =
\begin{cases}
\frac{-2 \alpha  a_1+\alpha  a_2-2 b_1}{a_2}-\frac{2 \left(\alpha  a_1+b_1\right)}{a_2 \left(e^{\left(z-z_0\right) \left(\alpha  a_1+b_1\right)}-1\right)}, b_2 = 2 \alpha  a_1-\alpha  a_2+2 b_1, \alpha  a_1+b_1 \not = 0,
\\
-\frac{2 \left(\alpha  a_1+b_1\right)}{a_2 \left(e^{\left(z-z_0\right) \left(\alpha  a_1+b_1\right)}-1\right)}-\frac{b_1}{a_1}, b_2 = \frac{-2 \alpha  a_1^2-2 a_1 b_1+a_2 b_1}{a_1},  \alpha  a_1+b_1 \not = 0,
\\
-\frac{\alpha  a_1+b_1 e^{\frac{a_2 \left(z-z_0\right) \left(\alpha  a_1+b_1\right)}{2 a_1}}}{a_1 \left(e^{\frac{a_2 \left(z-z_0\right) \left(\alpha  a_1+b_1\right)}{2 a_1}}-1\right)}, b_2 = \frac{a_2 b_1-\alpha  a_1 a_2}{2 a_1},  \alpha  a_1+b_1 \not = 0,
\\
-\frac{2}{a_2 \left(z - z_0\right)}-\frac{b_2}{a_2}, b_1 = -\alpha  a_1, b_2 = - \alpha  a_2  .
\end{cases}
\end{eqnarray*}
where $z_0$ is arbitrary.

\item [(C)] If neither $j_1$ nor $j_2$ is an integer, then  we conclude that
 all  meromorphic solutions of the ODE \eqref{decomposition-2factors-1} belong to the class $W$ and they can be found by using the same method as that in {\it Subcase B2}.
 % Moreover, they share the same form as the solutions in {\it Subcase B2}.
% case2

\end{itemize}

%\item
(II) $ a_1 a_2 \not   = 0,   a_2 = - 2 a_1$.

We consider a more general ODE which includes the equation \eqref{decomposition-2factors-2} and hence equation \eqref{decomposition-2factors-1}
\begin{eqnarray} \label{KPP}
u''(z)  + c u'(z)-\frac{2 }{\lambda ^2}\left(u(z)-q_1\right) \left(u(z)-q_2\right) \left(u(z)-q_3\right)= 0,
\end{eqnarray}
%and the correspondence among the coefficients in \eqref{decomposition-2factors-1} and \eqref{KPP} is
%  \begin{eqnarray*}
%  \begin{cases}
%  \lambda^2 = \dfrac{1}{a_1^2}  ,
%  c = \alpha a_1 - b_1 -b_2,
%  \\ \{
%  q_1, q_2, q_3 \} = \{\alpha, - \dfrac{b_1}{a_1},  \dfrac{b_2}{2 a_1}\}, \{
%  i, j, k \} = \{1, 2, 3\}
%  \end{cases}
%\end{eqnarray*}
where $ \lambda (\not = 0), c, q_1, q_2, q_3 \in \C$, for which the Fuchs indices are $-1, 4$ and the  compatibility conditions are
\begin{equation} \label{Compatibility Condition-KPP}
\begin{cases}
c \left(c \lambda +q_1+q_2-2 q_3\right) \left(c \lambda +q_1-2 q_2+q_3\right) \left(c \lambda -2 q_1+q_2+q_3\right) = 0, \, \text{if} \, u_{-1} = \lambda,
\\
c \left(c \lambda +2 q_1-q_2-q_3\right) \left(c \lambda -q_1+2 q_2-q_3\right) \left(c \lambda -q_1-q_2+2 q_3\right), \, \text{if} \, u_{-1} = -\lambda .
\end{cases}
\end{equation}

%Suppose equation \eqref{KPP} could be decomposed into the form

 Now we compare the ODE \eqref{KPP} with the equation \eqref{decomposition-2factors-1}. One can check that the compatibility conditions \eqref{Compatibility Condition-KPP} hold  if and only if   $c = 0$ or
  \begin{eqnarray*}
  \begin{cases}
  \lambda^2 = \dfrac{1}{a_1^2}  ,
  c = \alpha a_1 - b_1 -b_2 \not = 0,
  \\ \{
  q_1, q_2, q_3 \} = \{\alpha, - \dfrac{b_1}{a_1},  \dfrac{b_2}{2 a_1}\}.
%  \{  i, j, k \} = \{1, 2, 3\}.
  \end{cases}
\end{eqnarray*}
In other words, the compatibility conditions \eqref{Compatibility Condition-KPP} hold  if and only if   $c = 0$ or the equation \eqref{KPP} can be factorized into the form \eqref{decomposition-2factors-1}.

% For $c = 0$, the decomposition holds if and only if $2q_i - q_j - q_k = 0$ while meromorphic solutions exists for any choice of $q_i, a_j, q_k$ in this case.

 If $c = 0$, then the ODE \eqref{KPP} reduces to a first order Briot-Bouquet differential equation through multiplying it by $u'$ and performing an integration.
 Therefore all its meromorphic solutions  belong to the class $W$ and can be found by using the method introduced in \cite{Conte2010Ng,Eremenko2006}.

 For $c\not = 0$ and assuming $\eqref{Compatibility Condition-KPP}$ from now on, due to the symmetry in \eqref{Compatibility Condition-KPP} and the fact that $u(z)$ has at least one pole in $\C$, it suffices to consider the case $c = \dfrac{-q_1+2 q_2-q_3}{\lambda } \not = 0$ and one  choice for $a_i, b_i, i = 1, 2$ and $\alpha$ is
 \begin{eqnarray*}  \label{KPP-parameters of the decomposition}
a_1 = -\frac{1}{\lambda },  a_2 = \frac{2}{\lambda }, b_1 = \frac{q_3}{\lambda }, b_2 = -\frac{2 q_2}{\lambda },\alpha = q_1.
% \begin{cases}
% a_1 = -1, a_2 = 2, b_1 = q_3,  b_2 = -2 q_2, \alpha =  q_1, \\
%  a_1 = -1, a_2 = 2, b_1 = q_1,  b_2 = -2 q_3, \alpha = q_2, \\
% a_1 = -1, a_2 = 2, b_1 = q_2,  b_2 = -2 q_1, \alpha  = q_3.
% \end{cases}
\end{eqnarray*}

%satisfies \eqref{Compatibility Condition-KPP} with $u_{-1} = 1$

%As a consequence, if $c = \dfrac{-q_1+2 q_2-q_3}{\lambda } \not = 0$, then \eqref{KPP} can be written as
%\begin{eqnarray} \label{decomposition3}
%[D - \frac{2  }{\lambda }u -b_2] [D + \frac{u  }{\lambda } - b_1] (u - \alpha) = 0,
%\end{eqnarray}

Define $G(z), H(z), h(z)$ in the same way as in the case 1, that is $u = \dfrac{\lambda}{2} \dfrac{H'}{H}$, $ H(z) = e^{-b_2 z} h(z) $.
Then we have $u$ is meromorphic $\Leftrightarrow H$ is meromorphic $\Leftrightarrow  h$ is meromorphic.

%Let $G(z) = [D + \dfrac{u  }{\lambda } - b_1] (u - \alpha)$, then we have $[D - \dfrac{2  }{\lambda }u -b_2] G(z) = 0$ from which one can solve for $G(z) = \beta e^{\int \frac{2  }{\lambda } u dz} e^{b_2 z}  , \beta \in \C$. If $\beta = 0$, from $G(z) = [D + \frac{u  }{\lambda } - b_1] (u - \alpha) = 0$,
%we are able to obtain the first family of meromorphic solutions of the ODE \eqref{KPP}
%%and they are given by
%\begin{eqnarray} \label{pargicular solution-KPP}
%u(z) = \dfrac{q_1 e^{\frac{q_1 \left(z-z_0\right)}{\pm \lambda }}-q_3 e^{\frac{q_3 \left(z-z_0\right)}{ \pm \lambda }}}{e^{\frac{q_1 \left(z-z_0\right)}{ \pm \lambda }}-e^{\frac{q_3 \left(z-z_0\right)}{ \pm \lambda }}}, \, \, z_0 \in \C.
%\end{eqnarray}

For $\beta \not = 0$,
%by the substitution of $u = \dfrac{\lambda}{2} \dfrac{H'}{H}$ and \eqref{parameters of the decomposition} into \eqref{KPP},
using the same argument as in the case 1, we have
\begin{equation} \label{KPP-H}
-2 q_3  H H'+ \dfrac{4 q_1 q_3}{\lambda} H^2-4 \beta  e^{\frac{- 2 q_2 }{\lambda} z} H^3+2 \lambda  H H''-2 q_1  H H'-\lambda  H'^2=0,
\end{equation}
%\begin{eqnarray} \label{decomposition5}
%-2 b_1 \lambda  H(z) H'(z)+4 \alpha  b_1 H(z)^2-4 \beta  e^{b_2 z} H(z)^3+2 \lambda  H(z) H''(z)-2 \alpha  H(z) H'(z)-\lambda  H'(z)^2=0.
%\end{eqnarray}
%If we let $H(z) = e^{-b_2 z} h(z)$, then the ODE \eqref{decomposition5} reduces to
and
\begin{equation} \label{KPP-h}
-2 h  \left(\left(q_1 + q_3 - 2 q_2 \right) h' -\lambda  h'' \right)+ \dfrac{4  }{\lambda} (q_3 -  q_2) (q_1 -  q_2) h^2 -\lambda  h'^2-4 \beta  h ^3=0.
\end{equation}

Suppose $h(z)$ is a meromorphic solution of \eqref{KPP-h}, W.L.O.G, we assume that it has  a pole at $z = 0$ and $h(z) = \sum_{j=p}^{+\infty} h_j z^{j},
- p \in \N, h_p \not = 0$  then one can check that $p = -2$ and the Fuchs indices of the ODE \eqref{KPP-h} are $- 1, 4$ with the compatibility condition
%\begin{eqnarray} \label{Com Condition-h(z)-1}
%\left(\alpha +\left(b_1+b_2\right) \lambda \right){}^2 \left(-8 \alpha ^2+2 \alpha  \left(10 b_1+b_2\right) \lambda +\left(-8 b_1^2+2 b_2 b_1+b_2^2\right) \lambda ^2\right)=0,
%\end{eqnarray}
%which by the substitution of \eqref{KPP-parameters of the decomposition} reduces to
%for any choice of parameters in \eqref{parameters of the decomposition}
\begin{eqnarray} \label{KPP-Com Condition-h(z)-2}
 \left(q_1+q_2-2 q_3\right) \left(2 q_1-q_2-q_3\right) \left(q_1-2 q_2+q_3\right) = 0,
 \end{eqnarray}
which implies $q_3 = \frac{1}{2} \left(q_1+q_2\right)$ or $q_1 = \frac{1}{2} \left(q_2+q_3\right)$ since $c = \dfrac{2 q_2 - q_1 - q_3}{\lambda}  \not = 0 $.

Then by the substitution of \eqref{KPP-Com Condition-h(z)-2}, the ODE \eqref{KPP-h} reduces to
\begin{eqnarray*}
\begin{cases}
-\lambda ^2 h'(z)^2+\lambda  h(z) \left(2 \lambda  h''(z)+3 \left(q_2-q_1\right) h'(z)\right)
\\
+2 \left(q_2-q_1\right){}^2 h(z)^2-4 \beta  \lambda  h(z)^3=0,
\\
q_3 = \frac{1}{2} \left(q_1+q_2\right),
\end{cases}
\end{eqnarray*}
or
%&&\begin{cases}
%h(z) \left(2 h''(z)-h(z) \left(4 \beta  h(z)+\left(q_1-q_3\right){}^2\right)\right)-h'(z)^2=0,
%\\
%q_2 = \frac{1}{2} \left(q_1+q_3\right).
%\end{cases}
%\\
\begin{eqnarray*}
\begin{cases}
-\lambda ^2 h'(z)^2+\lambda  h(z) \left(2 \lambda  h''(z)+3 \left(q_2-q_3\right) h'(z)\right)
\\
+2 \left(q_2-q_3\right){}^2 h(z)^2-4 \beta  \lambda  h(z)^3= 0,
\\
q_1 = \frac{1}{2} \left(q_2+q_3\right).
\end{cases}
\end{eqnarray*}

Next, it suffices to consider the case $q_3 = \frac{1}{2} \left(q_1+q_2\right)$ due to the symmetry in the above two equations. By the translation against the dependent variable $u$, we may further assume $q_3 = 0$ which implies $q_1 + q_2 = 0$. Let us come back to equation \eqref{KPP-H}, which by the substitution of \eqref{KPP-Com Condition-h(z)-2} with $q_1 = -q_2 \not = 0, q_3 = 0$ reduces to
\begin{eqnarray*} \label{reduction1 of generalized fisher}
-\lambda  H'(z)^2+2 H(z) \left(\lambda  H''(z)+q_2 H'(z)\right)-4 \beta  H(z)^3 e^{-\frac{2 q_2 z}{\lambda }} = 0.
\end{eqnarray*}
Performing the transformation $H (z) = v(\zeta), \zeta = e^{-\frac{q_2}{\lambda }z } $ gives
%\begin{eqnarray} \label{reduction2 of generalized fisher}
%2 q_2^2 v\left(e^{-\frac{q_2 z}{\lambda }}\right) v''\left(e^{-\frac{q_2 z}{\lambda }}\right)-q_2^2 v'\left(e^{-\frac{q_2 z}{\lambda }}\right){}^2-4 \beta  \lambda  v\left(e^{-\frac{q_2 z}{\lambda }}\right){}^3
%\end{eqnarray}
\begin{eqnarray} \label{reduction2 of generalized fisher}
2  v  v'' - (v')^2 - \dfrac{4 \beta  \lambda} {q_2^2}  v^3 = 0, \, \, ' = \dfrac{d}{d \zeta}.
\end{eqnarray}
%which can be rewritten as
%\begin{eqnarray*} \label{reduction3 of generalized fisher}
%2 v \left[v'' - \dfrac{3 \beta  \lambda} {q_2^2} v^2 \right] - \left [(v')^2 - \dfrac{2 \beta  \lambda} {q_2^2}v^3 \right] = 0.
%\end{eqnarray*}
%Multiplying the above equation by $v'$ gives
%\begin{eqnarray} \label{reduction4 of generalized fisher}
%v  \varphi'(\zeta) -   v' \varphi(\zeta) = 0,
%\end{eqnarray}
%where $\varphi(\zeta) = (v')^2 - \dfrac{2 \beta  \lambda} {q_2^2}v^3.$
%Integrate the ODE \eqref{reduction4 of generalized fisher}, then
Upon integration of \eqref{reduction2 of generalized fisher},  we have
\begin{eqnarray*} \label{reduction5 of generalized fisher}
(v')^2 - \dfrac{2 \beta  \lambda} {q_2^2}v^3 + C v  = 0, C \in \C
\end{eqnarray*}
which has the general solution
\begin{eqnarray*} \label{reduction6 of generalized fisher}
v(\zeta) = \dfrac{2 q_2^2 } {\beta  \lambda} \wp(\zeta - \zeta_0; g_2, g_3),
\end{eqnarray*}
where $g_2 = \dfrac{C \beta  \lambda}{2 q_2^2},g_3 = 0, C, \zeta_0 \in \C$.

Finally, for  $c = \dfrac{- q_1 + 2 q_2  - q_3}{\lambda}  \not = 0$ and $ q_3 = \frac{1}{2} \left(q_1+q_2\right)$,
which implies $c = \dfrac{2 q_1 - q_2 - q_3}{- \lambda}$,  we obtain the meromorphic solution of the ODE \eqref{KPP}
(which meanwhile is the general solution)
%of the ODE \eqref{KPP}
\begin{equation*} \label{general solutions of generalized fisher}
u(z) = - \dfrac{  q_2 - q_3  } {2} e^{- \frac{ q_2 - q_3 }{\lambda} z}  \dfrac{\wp'(e^{- \frac{ q_2 - q_3 }{\lambda} z}  - \zeta_0; g_2, 0)}{\wp(e^{ - \frac{q_2 - q_3 }{\lambda} z}  - \zeta_0; g_2, 0)} + q_3, \, \, \zeta_0, g_2 \in \C.
\end{equation*}

%\begin{eqnarray}
%\begin{cases}
%c = \dfrac{2 q_1 - q_2  - q_3}{\lambda}  \not = 0,\\
%q_3 = \frac{1}{2} \left(q_1+q_2\right)
%\end{cases}
%\Leftrightarrow
%\begin{cases}
%c = \dfrac{- q_1 + 2 q_2  - q_3}{- \lambda}  \not = 0, \\
%q_3 = \frac{1}{2} \left(q_1+q_2\right)
%\end{cases}
%\end{eqnarray}

% case3
 % \item
 (III) $ a_1   = 0,   a_2 b_1 \not = 0$.

  We first look at the entire solution $u$ of the equation \eqref{decomposition-2factors-1}, for which we know that
  \begin{eqnarray} \label{decomposition-2factors-Case2-0-1}
\beta e^{\int (a_2 u + b_2)dz }  = u'  - b_1(u - \alpha), \beta \in \C.
\end{eqnarray}

For $\beta = 0$, the entire solutions are given by
\begin{eqnarray*} \label{decomposition-2factors-Case2-0-2}
u(z) =
\begin{cases}
\alpha + c e^{ b_1z},  b_1 \not = 0,
\\
c,  b_1 = 0,
% a_1 \not = 0,
%\\
%c, a_1 = \alpha  a_1+b_1 = 0,
\end{cases}
c \, \, \text{arbitrary}.
\end{eqnarray*}

For $\beta \not = 0$, we claim that the equation \eqref{decomposition-2factors-Case2-0-1} does not have any nonconstant entire solution.
Otherwise, suppose $u$ is a transcendental entire solution of \eqref{decomposition-2factors-Case2-0-1} as one can check immediately that \eqref{decomposition-2factors-Case2-0-1} does not admit any nonconstant polynomial solution. Let $U(z) = e^{\int (a_2 u + b_2)dz }$, then $U$ is transcendental entire and the equation \eqref{decomposition-2factors-Case2-0-1} becomes
  \begin{eqnarray*} \label{decomposition-2factors-Case2-0-3}
\beta e^{U}  = \dfrac{1}{a_2} \left( U'' - b_1 U' + b_1 b_2 + \alpha  b_1 a_2 \right).
\end{eqnarray*}
It implies that
  \begin{eqnarray*} \label{decomposition-2factors-Case2-0-4}
T(r, e^{U}) = O(T(r, U)),
\end{eqnarray*}
for all $r \in (0, + \infty)$ outside a possible exceptional set with finite linear measure, which contradicts Lemma \ref{Growth estimate1-factorizabl ODEs}.

%\begin{theorem}\cite[p.~210]{Chuang1990Yang} \label{Growth estimate1}
%Let $f$ and $g$ be two transcendental entire functions. Then
%\begin{eqnarray*}
%\lim_{r \rightarrow \infty} \dfrac{\log M(r, f \circ g)}{\log M(r, f)} = \infty, \, \lim_{r \rightarrow \infty} \dfrac{T (r, f \circ g)}{  T(r, f)} = \infty, \\
%\lim_{r \rightarrow \infty} \dfrac{\log M(r, f \circ g)}{\log M(r, g)} = \infty, \, \lim_{r \rightarrow \infty} \dfrac{T (r, f \circ g)}{  T(r, g)} = \infty.
%\end{eqnarray*}
%
%\end{theorem}

Next, we consider meromorphic solutions of \eqref{decomposition-2factors-1} with at least one pole on $\C$.  In this case, the ODE \eqref{decomposition-2factors-2} reduces to
\begin{equation} \label{decomposition-2factors-Case2-1}
  u'' + u' \left(  -a_2  u - b_1 - b_2\right)+ b_1 (u - \alpha )  \left(a_2 u+b_2\right) = 0,
\end{equation}
which has the Fuchs indices $j = - 1, 2$ with compatibility condition for the existence of meromorphic solutions:
\begin{equation*} \label{decomposition-2factors-Case2-2}
\alpha a_2 - 2 b_1 + b_2 = 0.
\end{equation*}
We make the change of variables $u \mapsto u + \alpha, b_1 \mapsto b'_1, b_2 \mapsto b'_2$, where $b'_1 = b_1 +  \alpha a_1, b'_2 = b_2 + 4  \alpha a_1$,
 %If well defined, in the subsequent argument, $H(z)$ and $h(z)$ will always refer to those given above.
  then with $b_2 = 2 b_1 \not = 0$, the ODE \eqref{decomposition-2factors-7} reduces to
\begin{equation*} \label{decomposition-2factors-Case2-3}
- a_2 \beta    e^{2 b_1 z} H(z)^3-b_1 H(z) H'(z)+H(z) H''(z)-H'(z)^2 = 0,\beta \not = 0.
\end{equation*}
Performing the transformation $H (z) = v(\zeta), \zeta = e^{b_1 z } $ gives
%\begin{eqnarray} \label{reduction2 of generalized fisher}
%2 q_2^2 v\left(e^{-\frac{q_2 z}{\lambda }}\right) v''\left(e^{-\frac{q_2 z}{\lambda }}\right)-q_2^2 v'\left(e^{-\frac{q_2 z}{\lambda }}\right){}^2-4 \beta  \lambda  v\left(e^{-\frac{q_2 z}{\lambda }}\right){}^3
%\end{eqnarray}
\begin{eqnarray} \label{decomposition-2factors-Case2-4}
\frac{a_2 \beta }{b_1^2} v^3-v v''+v'^2 = 0, \, \, ' = \dfrac{d}{d \zeta}.
\end{eqnarray}
%Multiplying $v v'$ on both sides of \eqref{decomposition-2factors-Case2-4}, it becomes
%\begin{eqnarray*} \label{decomposition-2factors-Case2-5}
%- v^2 v'  (v'' - \frac{3 a_2 \beta }{b_1^2} v^2) + 2 v v' \left(\dfrac{(v')^2}{2}  - \frac{a_2 \beta }{b_1^2} v^3 \right) = 0,
%\end{eqnarray*}
%which can be written as
%\begin{eqnarray} \label{decomposition-2factors-Case2-6}
%- v^2  \varphi'(\zeta) +   (v^2)' \varphi(\zeta) = 0,
%\end{eqnarray}
%where $\varphi(\zeta) = \dfrac{(v')^2}{2} - \dfrac{a_2 \beta }{b_1^2} v^3.$
%Integrate the ODE \eqref{decomposition-2factors-Case2-6}, then
Upon integration of \eqref{decomposition-2factors-Case2-4}, we have
\begin{eqnarray*} \label{decomposition-2factors-Case2-7}
\dfrac{(v')^2}{2} - \dfrac{a_2 \beta }{b_1^2} v^3 + c_0 v^2  = 0, c_0 \, \text{arbitrary},
\end{eqnarray*}
which has the general solution
\begin{eqnarray*} \label{decomposition-2factors-Case2-8}
v(\zeta) =
\begin{cases}
\frac{4 b_1^2}{a_2 \beta  \left(\zeta -c_1\right){}^2}, c_0 = 0,
\\
- \frac{b_1^2 }{ a_2 \beta} \frac{2 c_0^2}{  \cosh \left( \sqrt{2} c_0 \zeta + c_1  \right)+ 1}, c_0 \not = 0,
\end{cases}
 c_1 \, \text{arbitrary}.
\end{eqnarray*}
%where $g_2 = \dfrac{C \beta  \lambda}{2 q_2^2},g_3 = 0, C, \zeta_0 \in \C$.
After substitution, with \eqref{decomposition-2factors-Case2-2} satisfied, we obtain the meromorphic solution of the ODE \eqref{decomposition-2factors-Case2-1}
\begin{eqnarray*} \label{decomposition-2factors-Case2-8}
u(z) =
\begin{cases}
\alpha -\frac{2 b_1 e^{b_1 z}}{a_2 \left(e^{b_1 z}-c_1\right)}, c_0 = 0,
\\
\alpha -\frac{\sqrt{2} b_1 c_0 e^{b_1 z} \tanh \left(\frac{1}{2} \left(\sqrt{2} c_0 e^{b_1 z}+c_1\right)\right)}{a_2}, c_0 \not = 0,
\end{cases}
\end{eqnarray*}
where $c_0, c_1$ are arbitrary.
\begin{remark}
If $a_1 = 0$, the particular  solution \eqref{decomposition-2factors-paricular solution} is entire.
\end{remark}

%{\color{red}
%\begin{remark}
%For $ a_1   = 0,   a_2 b_1 \not = 0$, the solutions have been checked by Mathematica!
%\end{remark}
%}

%case4

%\item
(IV)  $ a_1 b_2 \not = 0, a_2 = 0$.

Upon the translation $u = w + \alpha$ and integration, the ODE \eqref{decomposition-2factors-1} reduces to a  Riccati equation
\begin{eqnarray} \label{decomposition-2factors-Case3-1}
\dfrac{d w}{d z} - a_1 w ^2- b'_1 w -\beta  e^{b_2 z} = 0,
\end{eqnarray}
where $\beta \in \C, b'_1 = b_1 + a_1 \alpha$.
It suffices to consider the case $\beta \not = 0$, otherwise the meromorphic solutions are given by \eqref{decomposition-2factors-paricular solution}. Denote by $w = -\dfrac{1}{a_1} \dfrac{v'}{v}$, then the equation \eqref{decomposition-2factors-Case3-1} is transformed to
\begin{eqnarray} \label{decomposition-2factors-Case3-2}
\dfrac{d^2 v}{d z^2} - b'_1 \dfrac{d v}{d z} + a_1 \beta  e^{b_2 z} v = 0.
\end{eqnarray}
According to  Lemma \ref{solutions on general linear ODEs}, all solutions of the ODE \eqref{decomposition-2factors-Case3-2} are entire functions and thus all the solutions of the ODE \eqref{decomposition-2factors-Case3-1} are meromorphic functions.

To find the general (meromorphic) solution of the ODE \eqref{decomposition-2factors-Case3-2}, we set
\begin{eqnarray*}
v(z) = e^{\frac{b'_1 z}{2}} f\left(\zeta\right), \zeta = \dfrac{2 \sqrt{a_1 \beta}   }{b_2} e^{\frac{b_2 z}{2}}.
\end{eqnarray*}
With the new variables, the equation  \eqref{decomposition-2factors-Case3-2} is transformed to the Bessel equation

\begin{equation*} \label{decomposition-2factors-Case3-3}
\zeta^2 \dfrac{d^2 f}{ d\zeta^2} + \zeta \dfrac{d f}{ d\zeta} + (\zeta^2 - \nu^2) f = 0, \nu = \dfrac{b'_1}{b_2},
\end{equation*}
which has the general solution
\begin{equation*} \label{decomposition-2factors-Case3-4}
f(\zeta) = c_1 J_{\nu}(\zeta) + c_2 Y_{\nu}(\zeta),
\end{equation*}
where $c_1, c_2 \in \C$ are arbitrary, $J_{\nu}(\zeta)$ and $Y_{\nu}(\zeta)$ are Bessel functions of the first second kinds respectively.
Consequently, for  $ a_1 b_2 \not = 0, a_2 = 0$, the general solution of \eqref{decomposition-2factors-1} which is meromorphic  is given by
\begin{equation*} \label{decomposition-2factors-Case3-5}
u(z) =   \frac{\alpha  a_1-b_1}{2 a_1}-\sqrt{\dfrac{\beta}{a_1}} \dfrac{e^{\frac{b_2 z}{2}} \left(c_1 J_{\nu}'\left(\zeta \right)+c_2 Y_{\nu}'\left(\zeta \right)\right)}{  \left(c_1 J_{\nu}\left(\zeta \right)+c_2 Y_{\nu}\left(\zeta \right)\right)},
\end{equation*}
where $\nu = \frac{\alpha  a_1 + b_1}{b_2}, \zeta = \dfrac{2 \sqrt{a_1 \beta}   }{b_2} e^{\frac{b_2 z}{2}}$ and $  \beta, c_1, c_2 \in \C$ are arbitrary.

%{\color{red} Question: when is $v$  zero-free so that $u$ is entire?}

\begin{remark}
Although $J_{\nu}\left(\zeta \right)$ and $Y_{\nu}\left(\zeta \right)$ as functions of $\zeta$ are not entire in general, $J_{\nu}\left( e^{z} \right)$ and $Y_{\nu}\left(e^{z} \right)$ as functions of $z$ are entire for any $\nu \in \C$.
\end{remark}

% case5
%\item
 (V) For  other cases,  the nonconstant meromorphic solutions given below of the ODE \eqref{decomposition-2factors-1} can be easily derived
   \begin{eqnarray*}
u(z) =
\begin{cases}
c_1 e^{b_1 z} +  c_2 e^{b_2 z}+ \alpha , \quad a_1 = a_2 = 0, b_1 \not = b_2,
\\
c_1 e^{b_1 z} +  c_2 z e^{b_1 z}+ \alpha  , \quad a_1 = a_2 = 0, b_1  = b_2,
\\
\frac{c_1 \cot \left(c_2-\frac{c_1 z}{2}\right)-b_2}{a_2}, \quad a_1 = b_1 = 0, a_2 \not = 0,
 \\
\frac{c_1 \cot \left(c_2-\frac{c_1 z}{2}\right) + \alpha  a_1-b_1}{2 a_1}, \quad a_1 \not = 0,  a_2 = b_2 = 0,
 \end{cases}
\end{eqnarray*}
where $c_1, c_2 \in \C$ are arbitrary.
%\end{enumerate}
\begin{remark}
The above solutions may degenerate to rational functions due to the degeneration of $c \cot ( c z)$ as $c$ approaches 0.
\end{remark}

Thus, the proof of Theorem \ref{solutions to 2-factorizble equations} is completed.

\section{Proof of Theorem \ref{growth estimate-2-factorizble equations}}

We first recall a lemma and some terminologies that  will be needed. For more details, see \cite{Kinnunen1998}.
\begin{lemma}  (\cite[p.~5]{Laine1993}) \label{growth comparsion--linear measure-factorizable ODEs}
Let $g:(0,+\infty)\rightarrow \R$ and $h:(0,+\infty)\rightarrow \R$ be monotone increasing functions such that $g(r) \leq h(r)$ outside of an exceptional set $F$ with finite linear measure. Then, for any $\alpha > 1$, there exists $r_0 > 0$ such that $g(r) < h(\alpha r )$ holds for all $r \geq r_0$.
\end{lemma}

%To estimate the order of growth of meromorphic solutions of the equation \eqref{decomposition-2factors-1} for this case, some terminologies are necessary. For more details, see \cite{Kinnunen1998}.
 The {\it iterated order} of a meromorphic function is defined by
\begin{equation*} \label{decomposition-2factors-Case3-6}
\rho_j (f) : = \limsup_{r \rightarrow \infty} \dfrac{\log_j T(r, f)}{\log r},
\end{equation*}
where $ \log_1(r) : = \log r, \log_j(r) : = \log \log_{j-1}(r) $. The {\it finiteness degree of growth} $i(f)$ of a meromorphic function $f$ is defined as
\begin{eqnarray*} \label{decomposition-2factors-Case3-7}
i(f) : =
\begin{cases}
0, \text{for} \, f \, \text{rational},
\\
\min\{j \in \N; \rho_j (f) < \infty\}, \text{for} \, f \, \text{transcendental}
\\
\infty, \text{otherwise}.
\end{cases}
\end{eqnarray*}
For the differential operator $L(f)$ defined by \eqref{solutions to linear DEs}
%\begin{eqnarray*} \label{decomposition-2factors-Case3-8}
% : = f^{(n)} + a_{n-1}(z)f^{(n-1)}+ \cdots + a_0(z) f
%\end{eqnarray*}
with entire coefficients, we define
\begin{eqnarray*} \label{decomposition-2factors-Case3-8}
&&\delta(L):= \max\{i(f); L(f) = 0\},
\\
&&\gamma_j(L):= \max\{\rho_j(f); L(f) = 0\}, j \in \N,
\\
&&p(L):= \max\{i(\alpha_j); j = 0, 1, \dots, n-1 \}.
%\\
%&&\kappa(L) := \max\{\rho_p(\alpha_j); j = 0, 1, \dots, n-1 \}, \text{if} \, 0 <  p: = p(L) < \infty.
\end{eqnarray*}
Another notation is defined for $0 <  p: = p(L) < \infty$,
\begin{eqnarray*} \label{decomposition-2factors-Case3-9}
\kappa(L) := \max\{\rho_p(\alpha_j); j = 0, 1, \dots, n-1 \}.
\end{eqnarray*}
Then we have
\begin{theorem} \cite{Kinnunen1998} \label{estimate on iterated order}
If $0 <  p  < \infty$, then $\delta(L) = p + 1$ and $\gamma_{p+1}(L) = \kappa(L) $. Moreover, if $\alpha_j$ is the last one in the sequence of coefficients $\alpha_0, \dots, \alpha_{n - 1}$ such that $i(\alpha_j) = p$, then the differential equation $L(f) = 0$ possesses at most $j$ linearly independent solutions $f$ such that $i(f) \leq p$.
\end{theorem}

\noindent {\it Proof of Theorem \ref{growth estimate-2-factorizble equations}}.
It has been shown in   \cite{Conte2015NgWu}  that for any $k \in \C$,  $ \exists \, \alpha_1, \beta_1 > 0$ such that
\begin{equation*}
\max\{T(r, \wp(e^{k z}  ) , T(r, \wp'(e^{k z}  ) \}<  \alpha_1 \exp(\beta_1 r ), 0 \leq r < \infty,
 \end{equation*}
 and the same method also gives that for any $k_1, k_2 \in \C$,  $ \exists \, \alpha_2, \beta_2 > 0$ such that
\begin{equation*}
T(r, \exp\{k_1 e^{k_2 z}\})<  \alpha_2 \exp(\beta_2 r ), 0 \leq r < \infty.
 \end{equation*}
On the other hand, for any function $f$ in the class $W$, we have $T(r, f) = O(r^2)$ or $o(r^2)$.

Then, according to the following properties of $T(r, f)$
\begin{eqnarray*}
T(r, \dfrac{\sum_{i=1}^n f_i}{\sum_{j=1}^m g_j}) &\leq& \sum_{i=1}^n T(r, f_i) + \sum_{j=1}^m T(r, g_j) + O(1),
\\
T(r, fg) &\leq& T(r, f)+ T(r, g),
 \end{eqnarray*}
% for any meromorphic  fucntions $$
we only need to consider   case  IV, because for other cases all the meromorphic solutions can be expressed as $u(z) = \dfrac{\sum f_i h_i}{\sum g_j y_j}$, where $f_i, h_i, g_j, y_j $  belong to either the class $W$ or
$\{\exp\{k_1 e^{k_2 z}\}, \wp(e^{k_3 z}), \wp'(e^{k_4 z}) | k_i \in \C, i = 1, 2, 3, 4\}$.

For case IV, to obtain an upper bound for $T(r, u)$, we only need to  estimate $T(r, v)$ because
\begin{eqnarray*}
T(r, u) &\leq& T(r, w) + O(1)
\\
&\leq& T(r, v) + T(r, v') + O(1)
\\
&\leq&  3 T(r, v) + S(r, v)
\\
&\leq& (3 + \varepsilon) T(r, v), \varepsilon >0
\end{eqnarray*}
for all $r \in (0,+ \infty) $   outside a possible exceptional set $E \subset (0,+ \infty) $ with finite linear measure.
%$T(r, \dfrac{f}{g}) \leq T(r, f) + T(r, g) + O(1), T(r, f') \leq 2 T(r, f) + S(r, f)$

For any $v$ satisfying \eqref{decomposition-2factors-Case3-2} with $p = 1, \alpha_0(z) =  a_1 \beta  e^{b_2 z}, \alpha_1(z) = -b_1$, Theorem \ref{estimate on iterated order} implies that $v$ is of infinite order, $\delta(L) = 2 $  and $\gamma_{2}(L) = \kappa(L) = 1$. As a consequence, $\rho_2(v) \leq 1$ and $T(r, v) \leq e^{r}$ for $r > 0$ sufficiently large. Thus we have
\begin{eqnarray*}
T(r, u) &\leq& (3 + \varepsilon) e^{r}, \varepsilon >0
\end{eqnarray*}
for all $r \in (0,+ \infty) \backslash E$. By Lemma \ref{growth comparsion--linear measure-factorizable ODEs}, we conclude that $T(r, u) \leq (3 + \varepsilon) e^{\alpha r}, \varepsilon >0, \alpha > 1$, for all sufficiently large $r$.

To show the sharpness of Hayman's conjecture, we consider the following three types of solutions of \eqref{decomposition-2factors-1}
\begin{eqnarray*}
u_1(z) &=& - \dfrac{  q_i - q_k  } {2} e^{- \frac{  q_i - q_k }{\lambda} z}  \dfrac{\wp'(e^{- \frac{ q_i - q_k }{\lambda} z}  - \zeta_0; g_2, 0)}{\wp(e^{ - \frac{ q_i - q_k }{\lambda} z}  - \zeta_0; g_2, 0)} + q_k, \, \,  g_2 \in \C,
\\
u_2(z) &=&   \frac{\alpha  a_1-b_1}{2 a_1}-\sqrt{\dfrac{\beta}{a_1}} \dfrac{e^{\frac{b_2 z}{2}} \left(c_1 J_{\nu}'\left(\zeta \right)+c_2 Y_{\nu}'\left(\zeta \right)\right)}{  \left(c_1 J_{\nu}\left(\zeta \right)+c_2 Y_{\nu}\left(\zeta \right)\right)}, \, \, c_1, c_2 \in \C,
\\
u_3(z) &=& \alpha -\frac{\sqrt{2} b_1 c_0 e^{b_1 z} \tanh \left(\frac{1}{2} \left(\sqrt{2} c_0 e^{b_1 z}+c_1\right)\right)}{a_2}, \, \, c_0, c_1 \in \C.
\end{eqnarray*}
Here, we choose $\nu = \frac{1}{2}, \zeta = e^{\frac{b_2}{2}z} $ for which
\begin{eqnarray*}
u_2(z) =
 \frac{\alpha  a_1-b_1}{2 a_1} + \dfrac{b_2}{4 a_1} \left(1+\frac{2 e^{\frac{b_2 z}{2}} \left(c_1 \cot \left(e^{\frac{b_2 z}{2}}\right)+c_2\right)}{ c_2 \cot \left(e^{\frac{b_2 z}{2}}\right)-c_1}\right).
\end{eqnarray*}
Then one can apply the same argument as that in \cite{Conte2015NgWu} to get lower bounds for the Nevanlinna counting functions $N(r, u_i), i = 1, 2, 3$, namely, there exist positive $\alpha_i, \beta_i, \gamma_i, i = 1, 2, 3$ such that
\begin{eqnarray*}
\alpha_i \exp\{\beta_i r^{\gamma_i}\} \leq N(r, u_i) \leq T(r, u_i).
\end{eqnarray*}
Thus, the proof is completed.

%\section*{Funding}
%
%The first author was partially supported by PROCORE - France/Hong Kong joint research grant, F-HK39/11T
% and RGC grant HKU 704409P. The second author was partially  supported by NSFC grant (No. 11701382) and RGC grant  HKU 704611P. % and a post-graduate studentship at HKU.

\section*{Acknowledgement}
We would like to thank A. E. Eremenko and Y. M. Chiang
for helpful discussions on the historical background of  Hayman's
conjecture and Robert Conte and the referee for their very valuable comments.

%\section*{Appendix}
%\newpage
%
%
%
%\thispagestyle{empty}
%table 1

\begin{center}

%\dummysidewaystable
\begin{sidewaystable}%[h]%{h}%[htpb]

\vspace{25em}
\section*{Appendix}
%\centering
\vspace{8em}
\resizebox{0.77\textwidth}{!}{\begin{minipage}{\textwidth}
\newcommand{\tabincell}[2]{\begin{tabular}{@{}#1@{}}#2\end{tabular}}
\centering % used for centering table
\caption{Meromorphic Solutions of ODE  \eqref{decomposition-2factors-1}.
%$[D - f_2(u)] [D - f_1(u)] (u - \alpha) = 0$,
%where $u = u(z), D = \dfrac{d}{d z}, \alpha \in \C$ and $f_i(u) = a_i u + b_i, a_i, b_i \in \C, i = 1, 2$.
In the following, $z_0, \zeta_0, c_0, c_1, c_2 \in \C$ are arbitrary.
} % title of Table
\begin{tabular}{ c | c | l  | c      } % centered columns (4 columns)
 \hline
 \hline %inserts double horizontal lines
\multicolumn{1}{c}{ } & \multicolumn{1}{c}{ }    &   \multicolumn{1}{c|}{Nonconstant meromorphic solutions   other than  \eqref{decomposition-2factors-paricular solution} }   & \multicolumn{1}{c}{Constraints on the parameters}    \\ %Method\#3 \\ [0.5ex] % inserts table
%heading
\cline{1-4} % inserts single horizontal line
{\multirow{7}{*}{\tabincell{c}{
$\begin{cases} a_1  a_2 \not = 0, \\
2 a_1 + a_2 \not = 0
%2-\frac{4a_1}{a_2} \not \in \N \cup \{0\}
\end{cases}$
  }}}
  %$ a_1 a_2 \not = 0$, \\ $2-\frac{4a_1}{a_2} \not \in \N \cup \{0\}$
  & {\multirow{5}{*}{\tabincell{c}{
$
 2-\frac{4a_1}{a_2} \not \in \N \cup \{0, -2\}
 $
  }}}&
$u(z) =\frac{-2 \alpha  a_1+\alpha  a_2-2 b_1}{a_2}-\frac{2 \left(\alpha  a_1+b_1\right)}{a_2 \left(e^{\left(z-z_0\right) \left(\alpha  a_1+b_1\right)}-1\right)}$
  &
  $b_2 = 2 \alpha  a_1-\alpha  a_2+2 b_1, \alpha  a_1+b_1 \not = 0  $
        \\   \cline{3-4}
 %Solution a_1 a_2 \not = 0, 2- (4a_1/a_2) \not \in \N   -2
          &  &
$u(z) = -\frac{2 \left(\alpha  a_1+b_1\right)}{a_2 \left(e^{\left(z-z_0\right) \left(\alpha  a_1+b_1\right)}-1\right)}-\frac{b_1}{a_1}$
  &
  $   b_2 = \frac{-2 \alpha  a_1^2-2 a_1 b_1+a_2 b_1}{a_1}, \alpha  a_1+b_1 \not = 0 $
        \\   \cline{3-4}
 %Solution a_1 a_2 \not = 0, 2- (4a_1/a_2) \not \in \N   -3
           &         &
$u(z) = -\frac{\alpha  a_1+b_1 e^{\frac{a_2 \left(z-z_0\right) \left(\alpha  a_1+b_1\right)}{2 a_1}}}{   a_1e^{\frac{a_2 \left(z-z_0\right) \left(\alpha  a_1+b_1\right)}{2 a_1}}- a_1 }$
  &
  $ b_2 = \frac{a_2 b_1-\alpha  a_1 a_2}{2 a_1} , \alpha  a_1+b_1 \not = 0 $
        \\   \cline{3-4}
 %Solution a_1 a_2 \not = 0, 2- (4a_1/a_2) \not \in \N   -4
    &       &

$u(z) = -\frac{2}{a_2 \left(z - z_0\right)}-\frac{b_2}{a_2}$
  &
  $ b_1 = -\alpha  a_1, b_2 = - \alpha  a_2$
        \\   \cline{2-4}

%\multicolumn{1}{  c  }{}

  %Solution a_1 a_2 \not = 0, 2- (4a_1/a_2)  \in \N   - 1

     &     \tabincell{c}{
  $\begin{cases}
%  a_1 a_2 \not = 0,
  2-\frac{4a_1}{a_2}   \in \N \cup \{0, -2\}, \\
2-\frac{a_2}{a_1 }  \not \in \Z
% \N \cup \{0\}
\end{cases}$
  }    &

 Not yet known
  &

        \\   \cline{1-4}

\end{tabular}
\label{table: 2nd order factorizable-1} % is used to refer this table in the text
\end{minipage} }
\end{sidewaystable}

%\begin{sidewaystable}[H]%{h}%[htpb]
%
%
%\resizebox{0.77\textwidth}{!}{\begin{minipage}{\textwidth}
%\newcommand{\tabincell}[2]{\begin{tabular}{@{}#1@{}}#2\end{tabular}}
%\centering % used for centering table
%% title of Table
%\begin{tabular}{ c        } % centered columns (4 columns)
% aaaaaaaaaaaaaaaaaaaaaaa
%
%\end{tabular}
%%\label{table: 2nd order factorizable-1} % is used to refer this table in the text
%\end{minipage} }
%\end{sidewaystable}

%\newpage
%
%\thispagestyle{empty}

% \begin{landscape}
%table 2
%\newpage
%\rotatebox{90}{
%\begin{rotate}{90}

\begin{sidewaystable}%[H]%[h]%{h}%[htpb]
\vspace{33em}
\resizebox{0.66\textwidth}{!}{
 \begin{minipage}{\textwidth}
%\centering
\newcommand{\tabincell}[2]{
\begin{tabular}{@{}#1@{}}#2\end{tabular}}
\caption{Meromorphic Solutions of ODE  \eqref{decomposition-2factors-1}.
%$[D - f_2(u)] [D - f_1(u)] (u - \alpha) = 0$,
%where $u = u(z), D = \dfrac{d}{d z}, \alpha \in \C$ and $f_i(u) = a_i u + b_i, a_i, b_i \in \C, i = 1, 2$.
In the following, $z_0, \zeta_0, c_0, c_1, c_2 \in \C$ are arbitrary.} % title of Table
%\centering % used for centering table
\begin{tabular}{ c | c | l | c  } % centered columns (4 columns)
\hline\hline %inserts double horizontal lines
&   & \multicolumn{1}{c|}{Nonconstant meromorphic solutions   other than  \eqref{decomposition-2factors-paricular solution} } & Constraints on the parameters \\ %Method\#3 \\ [0.5ex] % inserts table
%heading
\hline % inserts single horizontal line
% \tabincell{c}{$ a_1  a_2 \not = 0,$ \\ $2 a_1 + a_2 \not = 0$}
%  & $a_2 = 2 a_1$
%  & $a_2 =   a_1$
%  & $a_2 = 2 a_1$ &1
{\multirow{24}{*}{\tabincell{c}{$
\begin{cases} a_1  a_2 \not = 0, \\
2 a_1 + a_2 \not = 0,
\end{cases}$
\\
and
\\
$
\begin{cases}
2-\frac{4a_1}{a_2}   \in  \N \cup \{0, -2\}, \\
2-\frac{a_2}{a_1 }    \in \Z
%\N \cup \{0\}
\end{cases}$
}}}
   &
 $\begin{aligned} a_2 = 2   a_1 \quad  \; \;  \\
\Updownarrow \quad \quad  \quad \\
2-\frac{4 a_1}{a_2}  = 0 \end{aligned}$ & Nil &          \\ \cline{2-4}
%\multicolumn{1}{  c  }{}
&
%Solution a_2 =  a_1-1
{\multirow{8}{*}{\tabincell{c}{
%$\begin{eqnarray}
%a_2 =  a_1
%\Leftrightarrow
%2-\frac{4 a_1}{a_2} = -2
%\end{eqnarray}
%$
$\begin{aligned} a_2 =   a_1 \quad  \; \;  \\
\Updownarrow \quad \quad  \quad \\
2-\frac{4 a_1}{a_2}  = -2 \end{aligned}$
}}}  &
$u(z) = \frac{\left(b_1-b_2\right) \left(\alpha  a_1+b_2\right) \left(\alpha  a_1 c_2-b_1 c_1 e^{z \left(\alpha  a_1+b_1\right)}\right)-a_1 \beta  b_2 e^{z \left(\alpha  a_1+b_2\right)}}{a_1 \left(a_1 \left(\alpha  \left(b_1-b_2\right) \left(c_1 e^{z \left(\alpha  a_1+b_1\right)}+c_2\right)+\beta  e^{z \left(\alpha  a_1+b_2\right)}\right)+\left(b_1-b_2\right) b_2 \left(c_1 e^{z \left(\alpha  a_1+b_1\right)}+c_2\right)\right)}$
  &
  $\left(b_1 + \alpha a_1\right) \left(b_1-b_2\right) \left(\alpha  a_1+b_2\right) \not = 0$
      \\  \cline{3-4}
%Solution a_2 =  a_1-2
&   &
$u(z) = \dfrac{\alpha  \left(\alpha  \left(b_1-b_2\right){}^2 e^{b_1 z} \left(b_1 \left(c_2 z+c_1\right)+c_2\right)+\beta  b_1 b_2 e^{b_2 z}\right)}{b_1 \left(\alpha  \left(b_1-b_2\right){}^2 e^{b_1 z} \left(c_2 z+c_1\right)+\beta  b_1 e^{b_2 z}\right)}$
  &
  $b_1 + \alpha a_1 = 0, \left(b_1-b_2\right) \left(\alpha  a_1+b_2\right) \not = 0$
        \\  \cline{3-4}
%Solution a_2 =   a_1-3
  &   &
  $u(z) = \dfrac{e^{z \left(\alpha  a_1+b_2\right)} \left(b_2^2 c_2-a_1 \left(\beta +b_2 \left(\beta  z-\alpha  c_2\right)\right)\right)+\alpha  a_1 c_1}{a_1 \left(e^{z \left(\alpha  a_1+b_2\right)} \left(a_1 \left(\beta  z-\alpha  c_2\right)-b_2 c_2\right)+c_1\right)}$
  &
  $b_1 = b_2, \alpha  a_1+b_2 \not = 0$
     \\   \cline{3-4}
%Solution a_2 =   a_1-4
  &   &
  $u(z) = \dfrac{a_1 \left(\alpha  a_1 \left(\alpha  c_2+\beta  z\right)-\beta +\alpha  b_1 c_2\right)-b_1 c_1 e^{z \left(\alpha  a_1+b_1\right)}}{a_1 \left(c_1 e^{z \left(\alpha  a_1+b_1\right)}+a_1 \left(\alpha  c_2+\beta  z\right)+b_1 c_2\right)}$
  &
  $b_1 \not = b_2, \alpha  a_1+b_2 = 0$
     \\   \cline{3-4}
%Solution a_2 =   a_1-5
%\noalign{\smallskip}
    &   &
    $u(z) = \dfrac{-2 a_1 \left(\alpha  c_1+z \left(\beta +\alpha  c_2\right)\right)+\alpha  a_1^2 \beta  z^2+2 c_2}{a_1 \left(a_1 \beta  z^2-2 \left(c_2 z+c_1\right)\right)}$
  & $b_1 = b_2 = - \alpha a_1$    \\  \cline{2-4}
%Solution a_2 = - a_1-1
    &
{\multirow{4}{*}{\tabincell{c}{ $a_2 =  - a_1$}}}  &
\tabincell{c}{$u(z) =  \dfrac{12 \wp '\left(z - z_0;g_2,g_3\right)}{a_1 [\left(b_2-\alpha  a_1\right){}^2-12 \wp \left(z - z_0;g_2,g_3\right)]} + \dfrac{b_2}{a_1}  $}
  &
  \tabincell{c}{ $a_1 - b_1 - 2 b_2=0, $ \\ $ g_2 = \frac{1}{12} \left(b_2-\alpha  a_1\right){}^4,   g_3 \in \C $    }
      \\   \cline{3-4}
%Solution a_2 = - a_1-2
      &   &
      $u(z) = \frac{b_2}{a_1}+\frac{c \left(e_i-e_j\right) \left(2 e^{\frac{c z}{5}} \wp \left(e^{-\frac{1}{5} (c z)} - \zeta_0;0,g_3\right)+\wp '\left(e^{-\frac{1}{5} (c z)} - \zeta_0;0,g_3\right)\right)}{a_1 \left[ 5 \left(e_i-e_j\right) e^{\frac{c z}{5}} \wp \left(e^{-\frac{1}{5} (c z)} - \zeta_0;0,g_3\right)+5 e_j e^{\frac{3 c z}{5}}\right]},  \;   g_3 \in \C$
  &
   \tabincell{c}{ $c^2 \lambda  = 25 (e_i - e_j) \not = 0, i, j \in \{1, 2\}$, \\$
c = \alpha  a_1 - b_1 - 2 b_2,
\lambda = - \dfrac{6}{a_1 \beta } $,\\
$e_1 = 0, e_2 =  \left(b_1+b_2\right)  \left(\alpha  a_1-b_2\right)/(a_1 \beta)
$}   \\ \cline{2-4}
%Solution a_2 = - 4 a_1-1
  &
{\multirow{5}{*}{\tabincell{c}{ $a_2 = -4 a_1$}}}  &
$u(z) = -\dfrac{12 \wp '\left(z - z_0;g_2,g_3\right)}{2 a_1 \left(\left(\alpha  a_1+b_1\right){}^2-12 \wp \left(z- z_0;g_2,g_3\right)\right)}-\dfrac{b_1-\alpha  a_1}{2 a_1}$
  &
\tabincell{c}{ $b_2 = 2 \left(\alpha  a_1-b_1\right),$ \\ $ g_2 =   (b_1 + \alpha  a_1 ){}^4/12, \;   g_3 \in \C$    }
              \\ \cline{3-4}
%Solution a_2 = - 4 a_1-2
&   &
$u(z) = -\frac{c \left(e_2-e_1\right) \left[2 e^{\frac{c z}{5}} \wp \left(e^{-\frac{1}{5} (c z) } - \zeta_0;0,g_3\right)+\wp '\left(e^{-\frac{1}{5} (c z)} - \zeta_0;0,g_3\right)\right]}{2 a_1 \left[5 \left(e_2-e_1\right) e^{\frac{c z}{5}} \wp \left(e^{-\frac{1}{5} (c z)} - \zeta_0;0,g_3\right)+5 e_1 e^{\frac{3 c z}{5}}\right]}-\frac{b_1}{a_1}$
  & \tabincell{c}{ $ b_2 = 2 \left(3 \alpha  a_1+b_1\right), c = -  5 \left(\alpha  a_1+b_1\right) \not = 0$, \\ $ e_1 = 0,  e_2 =  3\left(\alpha  a_1+b_1\right)^2/(a_1 \beta),  \;  \beta \not = 0,   g_3 \in \C$ }    \\ \cline{3-4}
  %Solution a_2 = - 4 a_1-3
  &   &
  $u(z) = \alpha -\frac{c \left(e_2-e_1\right) \left[2 e^{\frac{c z}{5}} \wp \left(e^{-\frac{1}{5} (c z)} - \zeta_0;0,g_3\right)+\wp '\left(e^{-\frac{1}{5} (c z)} - \zeta_0;0,g_3\right)\right]}{2 a_1 \left[5 \left(e_2-e_1\right) e^{\frac{c z}{5}} \wp \left(e^{-\frac{1}{5} (c z)} - \zeta_0;0,g_3\right)+5 e_1 e^{\frac{3 c z}{5}}\right]}$
  & \tabincell{c}{ $ b_2 = -2 \left(\alpha  a_1+3 b_1\right), c =  5 \left(\alpha  a_1+b_1\right) \not = 0$, \\ $ e_1 = 0,  e_2 = 3\left(\alpha  a_1+b_1\right)^2/(a_1 \beta),  \;  \beta \not = 0,   g_3 \in \C$ }      \\ \cline{2-4}
  %Solution a_2 =  4 a_1-1
  &
  {\multirow{4}{*}{\tabincell{c}{ $a_2 = 4 a_1$}}}  &
  $u(z) = \alpha -\dfrac{1}{2 a_1 (z-c_0 )}-\dfrac{1}{2 a_1 (z-   c_1)}$
  &
  $\alpha a_1 + b_1 = 0$
      \\ \cline{3-4}
    %Solution a_2 =  4 a_1-2
&   &
$u(z) = \alpha -\dfrac{\left(\alpha  a_1+b_1\right) e^{z \left(\alpha  a_1+b_1\right)}}{2 a_1 \left(e^{z \left(\alpha  a_1+b_1\right)}+c_1\right)}$
  &
  $c_0 = 0, \alpha a_1 + b_1 \not = 0$
     \\ \cline{3-4}
    %Solution a_2 =  4 a_1-3
  &   &
  $u(z) = \alpha -\dfrac{c_0 \left(\alpha  a_1+b_1\right){}^3 e^{z \left(\alpha  a_1+b_1\right)} \left(e^{z \left(\alpha  a_1+b_1\right)}-c_1\right)}{a_1 \left(256 a_1 \beta +c_0 \left(\alpha  a_1+b_1\right){}^2 \left(e^{z \left(\alpha  a_1+b_1\right)}-c_1\right){}^2\right)}$
  &
  $c_0 \not = 0, \alpha a_1 + b_1 \not = 0$    \\ \cline{1-4}
\end{tabular}
\label{table: 2nd order factorizable-2} % is used to refer this table in the text
\end{minipage} }
\end{sidewaystable}

%\rotatebox{90}{

% \end{landscape}
%\end{rotate}
%\newpage
%table 3

%\newpage
%
%\thispagestyle{empty}

\begin{sidewaystable}%[h]%{h}%[H]
%\centering
\vspace{33em}
\resizebox{0.68\textwidth}{!}{\begin{minipage}{\textwidth}
\newcommand{\tabincell}[2]{\begin{tabular}{@{}#1@{}}#2\end{tabular}}
\caption{Meromorphic Solutions of ODE  \eqref{decomposition-2factors-1}.
%$[D - f_2(u)] [D - f_1(u)] (u - \alpha) = 0$,
%where $u = u(z), D = \dfrac{d}{d z}, \alpha \in \C$ and $f_i(u) = a_i u + b_i, a_i, b_i \in \C, i = 1, 2$.
In the following, $z_0, \zeta_0, c_0, c_1, c_2 \in \C$ are arbitrary.
} % title of Table
\centering % used for centering table
\begin{tabular}{ c | c | l  | c      } % centered columns (4 columns)
 \hline
 \hline %inserts double horizontal lines
\multicolumn{1}{c}{ }  & \multicolumn{1}{c}{ }   &   \multicolumn{1}{c|}{Nonconstant meromorphic solutions   other than  \eqref{decomposition-2factors-paricular solution} }   & \multicolumn{1}{c}{Constraints on the parameters}    \\ %Method\#3 \\ [0.5ex] % inserts table
%heading
\cline{1-4} % inserts single horizontal line
 %Solution  a_2   = -2 a_1--1
{\multirow{ 20}{*}{\tabincell{c}{
$\begin{cases} a_1  a_2 \not = 0, \\
  2 a_1 + a_2  = 0
\end{cases}$
\\ and set $  \lambda^2 = \dfrac{1}{a_1^2} $, \\ $c = \alpha a_1 - b_1 -b_2$, \\ $\{
  q_1, q_2, q_3 \} = \{\alpha, - \dfrac{b_1}{a_1},  \dfrac{b_2}{2 a_1}\}$,
  \\$\{
  i, j, k \} = \{1, 2, 3\}$  }}}
  & {\multirow{16}{*}{ $c = 0$}}
      &
$u(z) = \frac{3 \lambda ^2}{\left(z-z_0\right) \left[\left(q_j-q_k\right) \left(z-z_0\right) \pm 3 \lambda \right]}+q_j$
  & $q_i =  q_j$
%  {\multirow{ 6}{*}{\tabincell{c}{$ \lambda^2 = \dfrac{1}{a_1^2} $, \\ $c = \alpha a_1 - b_1 -b_2$, \\ $\{
%  q_1, q_2, q_3 \} = \{\alpha, - \frac{b_1}{a_1},  \frac{b_2}{2 a_1}\}$  }}}
  \\  \cline{3-4}
   %Solution  a_2   = -2 a_1--2
  & & $u(z) = \pm \lambda  m_1 \cot \left[m_1 \left(z-z_0\right)\right]+\frac{1}{3} \left(q_1+q_2+q_3\right)$
   &   $q_k =\dfrac{q_i + q_j}{2}, m_1 = \dfrac{\sqrt{-1}}{2 \lambda } (q_i-q_j) \not = 0   $
  \\  \cline{3-4}
   %Solution  a_2   = -2 a_1--3
  & & $ u(z) =  \lambda  m_2 \left( \cot \left[m_2 \left(z-z_0\right)\right] -   \cot \left[m_2 \left(z-z_0-a \right)\right] \right) + h$  &
$\begin{cases}
h = q_i, m_2 = \pm \dfrac{1}{\sqrt{2} \lambda } \sqrt{-(q_j - q_i) (q_k - q_i)} \not = 0,
\\
 m_2 \cot m_2 a = \dfrac{q_j+q_k-2q_i}{3 \lambda }
\end{cases}$
  \\  \cline{3-4}
   %Solution  a_2   = -2 a_1--4
  &  &
  \tabincell{l}
  {$u(z) = \dfrac{\lambda \wp'(a)}{\wp(z-z_0) - \wp(a)} + h,$  $h \in \C,
$  }   &
$
\begin{cases}
\wp(a)  =  \frac{1}{6 \lambda ^2}\left(3 h^2-2 h s_1+s_2\right),
\\
\wp'(a)  =   \frac{1}{\lambda ^3}\left(h-q_1\right) \left(h-q_2\right) \left(h-q_3\right),
\\
g_2  =  \frac{1}{3 \lambda ^4}\left(-3 h^4+4  s_1 h^3-6 s_2 h^2+12 s_3 h
+s_2^2  - 4 s_1 s_3\right),
\\
g_3   =  \frac{1}{27 \lambda ^6} \left[3 \left(s_1^2- 3 s_2\right)h^4  - 4 s_1 \left(s_1^2 - 3s_2\right)h^3
 + 6 s_2 \left(s_1^2- 3 s_2\right)h^2 \right.
 \\
\quad   \left. -12 s_3 \left(s_1^2- 3 s_2\right) h - s_2^3 + 6 s_1 s_2 s_3 - 27 s_3^2 \right],
 \\
 s_1 =q_1+q_2+q_3, \quad
s_2=q_2 q_3+q_3 q_1+q_1 q_2, \quad
s_3=q_1 q_2 q_3.
\end{cases}
$
%with the notation for the symmetric polynomials
%\begin{eqnarray*}
%s_1=q_1+q_2+q_3,\
%s_2=q_2 q_3+q_3 q_1+q_1 q_2,\
%s_3=q_1 q_2 q_3.
%\end{eqnarray*}
  %$
%   \begin{cases}
%g_2 = \frac{1}{3 \lambda ^4}[-3 h^4+4  \left(q_1+q_2+q_3\right)h^3-6  \left(q_1 q_2+q_3 q_2+q_1 q_3\right)h^2+12  q_1 q_2 q_3 h
%\\
%+q_1^2 q_2^2 +q_1^2 q_3^2 + q_2^2 q_3^2 - 2 \left(q_1^2 q_2 q_3 + q_1 q_2^2 q_3 + q_1 q_2 q_3^2 \right)],
%\\
%g_3  = \frac{1}{27 \lambda ^6} [3 \left(q_1^2+q_2^2+q_3^2- q_1 q_2-q_1 q_3-q_2 q_3\right)h^4  - 4 \left(q_1^3-3 q_1 q_2 q_3+q_2^3+q_3^3\right)h^3
%\\
%+ 6 \left(q_1^2+q_2^2+q_3^2-q_1 q_2 - q_1 q_3 -q_2 q_3\right)(q_1 q_2+q_1 q_3+q_2 q_3)h^2
% \\
% -12 q_1 q_2 q_3 \left(q_1^2+q_2^2+q_3^2-q_1 q_2 - q_1 q_3 -q_2 q_3\right) h -q_1^3 q_2^3 -q_1^3 q_3^3-q_2^3 q_3^3+3 q_2 q_3^2 q_1^3
% \\
% +3 q_2^2 q_3 q_1^3 +3 q_2^3 q_3 q_1^2+3 q_2^3 q_3^2 q_1+3 q_2^2 q_3^3 q_1+3 q_2 q_3^3 q_1^2-15 q_1^2 q_2^2 q_3^2 ]
%\end{cases}
%$

%  $q_1^2+q_2^2+q_3^2- q_1 q_2 - q_1 q_3 - q_2 q_3= 0$
  \\  \cline{2-4}
   %Solution  a_2   = -2 a_1--5
  & {\multirow{3 }{*}{ $c \not = 0$}} &
  $u(z) = \dfrac{q_j e^{\frac{q_j \left(z-z_0\right)}{\pm \lambda }}-q_k e^{\frac{q_k \left(z-z_0\right)}{ \pm \lambda }}}{e^{\frac{q_j \left(z-z_0\right)}{ \pm \lambda }}-e^{\frac{q_k \left(z-z_0\right)}{ \pm \lambda }}}$ &
  $c = \dfrac{2 q_i - q_j - q_k}{\pm \lambda }  \not = 0 $
  \\  \cline{3-4}
   %Solution  a_2   = -2 a_1--6
  &  & $u(z) = - \dfrac{  q_i - q_k  } {2} e^{- \frac{  q_i - q_k }{\lambda} z}  \dfrac{\wp'(e^{- \frac{ q_i - q_k }{\lambda} z}  - \zeta_0; g_2, 0)}{\wp(e^{ - \frac{ q_i - q_k }{\lambda} z}  - \zeta_0; g_2, 0)} + q_k$ &
  $  g_2 \in \C, c = \dfrac{2 q_i - q_j - q_k}{  \lambda } = \dfrac{- q_i + 2 q_j - q_k}{  - \lambda } \not = 0$
  \\  \cline{1-4}
 %Solution a_1 a_2 \not = 0, 2- (4a_1/a_2) \not \in \N   -1

%Solution    a_1 b_2 \not = 0, a_2 = 0
$  a_1  \not = 0, a_2 = 0, b_2 \not = 0$   &
& $u(z) =   \frac{\alpha  a_1-b_1}{2 a_1}-\sqrt{\dfrac{\beta}{a_1}} \dfrac{e^{\frac{b_2 z}{2}} \left(c_1 J_{\nu}'\left(\zeta \right)+c_2 Y_{\nu}'\left(\zeta \right)\right)}{  \left(c_1 J_{\nu}\left(\zeta \right)+c_2 Y_{\nu}\left(\zeta \right)\right)}$ &
$ \nu = \frac{\alpha  a_1 + b_1}{b_2}, \zeta = \dfrac{2 \sqrt{a_1 \beta}   }{b_2} e^{\frac{b_2 z}{2}},  \beta, c_1, c_2 \in \C $
   \\   \cline{1-4}

%Solution a_1 \not = 0,  a_2 = b_2 = 0
$a_1  \not = 0, a_2 = 0, b_2   = 0 $  &
 &
$u(z) = \frac{c_1 \cot \left(c_2-\frac{c_1 z}{2}\right) + \alpha  a_1-b_1}{2 a_1}$
 &

      \\   \cline{1-4}
%Solution  a_1   = 0,   a_2 b_1 \not = 0 ----1
{\multirow{2}{*}{$ a_1   = 0, a_2 \not = 0, b_1 \not = 0$ }}
& &
$u(z) =\alpha -\frac{2 b_1 e^{b_1 z}}{a_2 \left(e^{b_1 z}-c_1\right)}$
  &
  $c_0 = 0,   b_2 = - \alpha a_2 + 2 b_1  $
        \\   \cline{2-4}
%Solution  a_1   = 0,   a_2 b_1 \not = 0 ----2

     & &
  $u(z) = \alpha -\frac{\sqrt{2} b_1 c_0 e^{b_1 z} \tanh \left(\frac{1}{2} \left(\sqrt{2} c_0 e^{b_1 z}+c_1\right)\right)}{a_2}$
  &
  $c_0 \not = 0,   b_2 = -  \alpha  a_2 + 2 b_1  $
     \\
     \cline{1-4}
%Solution a_1 = b_1 = 0, a_2 \not = 0
 $ a_1   = 0, a_2 \not = 0, b_1  = 0$
 & &
  $u(z) =\frac{c_1 \cot \left(c_2-\frac{c_1 z}{2}\right)-b_2}{a_2}$
  &

     \\   \cline{1-4}
%Solution a_1 = a_2 = 0, b_1 \not = b_2
  $a_1 =0,  a_2 = 0, b_1 \not = b_2$
  & &
    $u(z) =c_1 e^{b_1 z} +  c_2 e^{b_2 z}+ \alpha$ &
    \\   \cline{1-4}
%Solution a_1 = a_2 = 0, b_1  = b_2
  $a_1 =0, a_2 = 0, b_1  = b_2$
  & &
 \tabincell{c}{$u(z) = c_1 e^{b_1 z} +  c_2 z e^{b_1 z}+ \alpha  $}
&
      \\   \cline{1-4}

\end{tabular}
\label{table: 2nd order factorizable-3} % is used to refer this table in the text
\end{minipage} }
\end{sidewaystable}

\end{center}

%\newpage

%\section{Bibliography styles}
%
%There are various bibliography styles available. You can select the style of your choice in the preamble of this document. These styles are Elsevier styles based on standard styles like Harvard and Vancouver. Please use Bib\TeX\ to generate your bibliography and include DOIs whenever available.
%
%Here are two sample references: \cite{Feynman1963118,Dirac1953888}.
%
%\section*{References}

\end{document}